\documentclass[reqno]{amsart}
\usepackage{bbm, dsfont}
 \usepackage{amsaddr}
\newtheorem{thm}{Theorem}[section]
\newtheorem{lm}{Lemma}[section]
\newtheorem{df}{Definition}[section]
\newtheorem{rmk}{Remark}[section]
\newtheorem{prop}{Proposition}[section]
\newtheorem{cor}{Corollary}[section]
\numberwithin{equation}{section}
\usepackage{amssymb}
\usepackage{bbm}
\usepackage{hyperref}
\usepackage{amsmath}
\begin{document}
\title[Stochastic Parabolic Equations with Dynamic Boundary Conditions]{Null Controllability for Backward Stochastic Parabolic Convection-Diffusion Equations with Dynamic Boundary Conditions}


\author{Mahmoud BAROUN $^1$, SAID BOULITE $^2$, ABDELLATIF ELGROU $^1$, AND LAHCEN MANIAR $^1$}
\address{$^1$ Cadi Ayyad University, Faculty of Sciences Semlalia, LMDP, UMMISCO (IRD-UPMC), P.B. 2390, Marrakesh, Morocco}
\email{m.baroun@uca.ac.ma, abdellatif.elgrou@ced.uca.ma, maniar@uca.ma}
\address{$^2$ Cadi Ayyad University, National School of Applied Sciences, LMDP, UMMISCO (IRD-UPMC), P.B. 575, Marrakesh,
Morocco}

\email[Corresponding author]{Corresponding author :s.boulite@uca.ma}
\thanks{}



\date{}

\dedicatory{}

\begin{abstract}
This paper is concerned with the null controllability for linear backward stochastic parabolic equations with dynamic boundary conditions and convection terms. Using the classical duality argument, the null controllability is obtained via an appropriate observability inequality of the corresponding adjoint forward stochastic parabolic equation. To prove this observability inequality, we develop a new global Carleman estimate for forward stochastic parabolic equations that contains some first-order terms in the weak divergence form. Our Carleman estimate is established by applying the duality technique. Moreover, an estimate of the null-control cost is provided.
\end{abstract}
\maketitle
\section{Introduction and main results}
The controllability of partial differential equations (PDEs for short) has been discussed in many works during the last years including parabolic equations, see, e.g., \cite{fernandez2006null,fernandez2006global,fursikov1999optimal,fursikov1996controllability,GloLions,imanuvilov2003carleman,khoutaibi2020null,lebeau1995controle,lions1972some,maniar2017null,russell73,zuazua2007controllability}, and the field is still highly active in the research. In the literature, we also find numerous works and results on the controllability for stochastic partial differential equations (SPDEs for short), particularly stochastic parabolic equations, see for instance \cite{barbu2003carleman,elgrouDBC,elgrouDBC1d,liu2014global,liu2019carleman,lu2011some,lu2021mathematical,tang2009null,yan2018carleman,zhang2008unique}, and the references cited therein. The controllability results for some classes of stochastic parabolic equations are studied with different kinds of boundary conditions, including Dirichlet, Neumann, and Fourier boundary conditions, see \cite{Preprintelgrou23,lu2021mathematical,tang2009null,yan2018carleman}, and with dynamic boundary conditions without the presence of convection terms, see \cite{elgrouDBC,elgrouDBC1d}. Moreover, one can also find some interesting controllability results for degenerate, fourth order, semilinear, coupled, and singular stochastic parabolic equations, see, e.g., \cite{fadili,liu2019carleman,fourthorder,san23,yansing}.

By the duality argument, one can show that the null controllability problem of linear PDEs or SPDEs is equivalent to an observability inequality for the corresponding adjoint equation. To prove this observability inequality, one of the key tools is Carleman estimates which are some energy estimates with exponential weights and some parameters that can be taken large enough to absorb some undesired and lower-order terms. These Carleman estimates were introduced by T. Carleman in 1939 to prove the uniqueness of solutions to second-order elliptic PDEs with two variables, for that see \cite{Carl39}. Now, such estimates become very useful for studying various topics in applied mathematics such as controllability (e.g., null controllability, exact controllability, approximate controllability, and so on), observability, stabilization, inverse problems, and the unique continuation property, for both PDEs and SPDEs, see, for instance, \cite{lu2012car,lu2021mathematical,zhang2008unique, zuazua2007controllability}. 

In this paper, we prove a new global Carleman estimate for forward stochastic parabolic equations containing some weak divergence terms. This estimate will prove the null controllability for general backward stochastic parabolic equations with dynamic boundary conditions and some convection terms. The approach used to establish our Carleman estimate is based on the duality technique, see \cite{imanuvilov2003carleman,liu2014global} for more details.

Let $T>0$, $G\subset\mathbb{R}^N$ be a given nonempty bounded domain with a $C^2$ boundary $\Gamma=\partial G$, $N\geq2$, let $G_0\Subset G$ be a given non-empty open subset which is strictly contained in $G$ (i.e., $\overline{G_0}\subset G$) and  $\overline{G}$ denotes the closure of $G$. Put $Q=(0,T)\times G, \,\,\, \Sigma=(0,T)\times\Gamma,\,\,\, \text{and}\,\,\, Q_0=(0,T)\times G_0$. Also, we indicate by $\mathbbm{1}_{G_0}$ the characteristic function of $G_0$.

Let $(\Omega,\mathcal{F},\{\mathcal{F}_t\}_{t\geq0},\mathbb{P})$ be a fixed complete filtered probability space on which a one-dimensional standard Brownian motion $W(\cdot)$ is defined such that $\{\mathcal{F}_t\}_{t\geq0}$ is the natural filtration generated by $W(\cdot)$ and augmented by all the $\mathbb{P}$-null sets in $\mathcal{F}$. Let $\mathcal{X}$ be a Banach space, let $C([0,T];\mathcal{X})$ be the Banach space of all $\mathcal{X}$-valued continuous functions defined on $[0,T]$; and for some sub-sigma algebra $\mathcal{G}\subset\mathcal{F}$, we denote by $L^2_{\mathcal{G}}(\Omega;\mathcal{X})$ the Banach space of all $\mathcal{X}$-valued $\mathcal{G}$-measurable random variables $X$ such that $\mathbb{E}\big(\vert X\vert_\mathcal{X}^2\big)<\infty$, with the canonical norm; by $L^2_\mathcal{F}(0,T;\mathcal{X})$ the Banach space consisting of all $\mathcal{X}$-valued $\{\mathcal{F}_t\}_{t\geq0}$-adapted processes $X(\cdot)$ such that $\mathbb{E}\big(\vert X(\cdot)\vert^2_{L^2(0,T;\mathcal{X})}\big)<\infty$, with the canonical norm; by $L^\infty_\mathcal{F}(0,T;\mathcal{X})$ the Banach space consisting of all $\mathcal{X}$-valued $\{\mathcal{F}_t\}_{t\geq0}$-adapted essentially bounded processes with its norm is denoted by $|\cdot|_\infty$; and by $L^2_\mathcal{F}(\Omega;C([0,T];\mathcal{X}))$ the Banach space consisting of all $\mathcal{X}$-valued $\{\mathcal{F}_t\}_{t\geq0}$-adapted continuous processes $X(\cdot)$ such that $\mathbb{E}\big(\vert X(\cdot)\vert^2_{C([0,T];\mathcal{X})}\big)<\infty$, with the canonical norm. Similarly, one can define $L^\infty_\mathcal{F}(\Omega;C^m([0,T];\mathcal{X}))$ for any positive integer $m$.
	
Consider the following Hilbert space
$$\mathbb{L}^2:=L^2(G,dx)\times L^2(\Gamma,d\sigma),$$
which $dx$ (resp., $d\sigma$ ) denoted the Lebesgue measure (resp., surface measure) on $G$ (resp., $\Gamma$), equipped with the following inner product 
$$\langle(y,y_\Gamma),(z,z_\Gamma)\rangle_{\mathbb{L}^2} = ( y,z)_{L^2(G)} + (y_\Gamma,z_\Gamma)_{L^2(\Gamma)}.$$
The main purpose of this paper is to establish the null controllability of the following backward stochastic parabolic equation with dynamic boundary conditions:
\begin{equation}\label{1.1}
{\small\begin{cases}
	\begin{array}{ll}
				dy + \text{div}(A\nabla y) \,dt = (a_1 y+a_2 Y +B\cdot\nabla y +\mathbbm{1}_{G_0}u) \,dt + Y \,dW(t) &\textnormal{in}\,\,Q,\\
				dy_\Gamma+\text{div}_\Gamma(A_\Gamma\nabla_\Gamma y_\Gamma) \,dt-\partial_\nu^A y \,dt = (b_1y_\Gamma+b_2\widetilde{Y}+B_\Gamma\cdot\nabla_\Gamma y_\Gamma)\,dt+\widetilde{Y} \,dW(t) &\textnormal{on}\,\,\Sigma,\\
				y_\Gamma(t,x)=y\vert_\Gamma(t,x) &\textnormal{on}\,\,\Sigma,\\
				(y,y_\Gamma)\vert_{t=T}=(y_T,y_{\Gamma,T}) &\textnormal{in}\,\,G\times\Gamma,
			\end{array}
		\end{cases}}
\end{equation}
where $a_1, a_2\in L_\mathcal{F}^\infty(0,T;L^\infty(G))$, $B\in L_\mathcal{F}^\infty(0,T;L^\infty(G;\mathbb{R}^N))$, $b_1, b_2\in L_\mathcal{F}^\infty(0,T;L^\infty(\Gamma))$, $B_\Gamma\in L_\mathcal{F}^\infty(0,T;L^\infty(\Gamma;\mathbb{R}^N))$, $(y_T,y_{\Gamma,T})\in L^2_{\mathcal{F}_T}(\Omega;\mathbb{L}^2)$ is the terminal state, $u \in L^2_\mathcal{F}(0,T;L^2(G_0))$ is the control process and $(y,y_\Gamma,Y,\widetilde{Y})$ is the state variable.
	
We assume that the diffusion matrices $A=(a_{i,j})_{i,j=1,2,...,N}$ and $A_\Gamma=(b_{i,j})_{i,j=1,2,...,N}$ satisfy the following assumptions:
\begin{enumerate}
    \item $A\in L^\infty_\mathcal{F}(\Omega;C^1([0,T];W^{1,\infty}(\overline{G};\mathbb{R}^{N\times N})))$, and $a_{i,j}=a_{j,i}$ \,\,for all \,\,$i,j=1,2,...,N$.
    \item $A_\Gamma \in L^\infty_\mathcal{F}(\Omega;C^1([0,T];W^{1,\infty}(\Gamma;\mathbb{R}^{N\times N})))$, and $b_{i,j}=b_{j,i}$ \,\,for all \,\,$i,j=1,2,...,N$.
    \item There exists a constant $\beta>0$ such that
    $$\langle A(\omega,t,x)\xi,\xi\rangle_{\mathbb{R}^N} \geq \beta\vert\xi\vert^2,\qquad (\omega,t,x,\xi)\in\Omega\times \overline{Q}\times\mathbb{R}^N,$$
    $$\langle A_\Gamma(\omega,t,x)\xi,\xi\rangle_\Gamma \geq \beta\vert\xi\vert^2,\qquad (\omega,t,x,\xi)\in\Omega\times\Sigma\times\mathbb{R}^N,$$
\end{enumerate}
where $\langle\cdot,\cdot\rangle_{\mathbb{R}^N}$ is the Euclidean inner product on $\mathbb{R}^N$ and $\langle\cdot,\cdot\rangle_\Gamma$ is the Riemannian inner product on $\Gamma$, which both will be denoted by $\textbf{a}\cdot \textbf{b}=\langle \textbf{a},\textbf{b}\rangle_{\mathbb{R}^N}=\langle \textbf{a},\textbf{b}\rangle_{\Gamma}$. Here $y\vert_\Gamma$ denotes the trace of $y$, $\partial^A_\nu y = (A\nabla y \,\cdot\, \nu)\vert_\Gamma$ designates the co-normal derivative w.r.t $A$ where $\nu$ is the outer unit normal vector at $\Gamma$. This co-normal derivative couples the bulk and surface equations. The term $\text{div}$ (resp., $\text{div}_\Gamma$) denotes the divergence (resp., tangential divergence) operator w.r.t to the space variable in $G$ (resp., $\Gamma$). Note that when $A=I_N$ the identity matrix, the co-normal derivative becomes only the standard normal derivative $\partial_\nu y=(\nabla y\cdot\nu)\vert_\Gamma$ and $\text{div}(A\nabla y)=\Delta y$ the Laplacian operator of $y$. The term $\text{div}_\Gamma(A_\Gamma\nabla_\Gamma y_\Gamma)$ represents surface diffusion effects on the boundary $\Gamma$ with $\nabla_\Gamma y_\Gamma:=\nabla y - (\partial_\nu y)\nu,$ (where $y$ is an extension of $y_\Gamma$ up to an open neighborhood of $\Gamma$) is called the tangential gradient of $y_\Gamma$ on $\Gamma$. See also that if $A_\Gamma=I_N$, we have $\text{div}_\Gamma(A_\Gamma\nabla_\Gamma y_\Gamma)=\Delta_\Gamma y_\Gamma$, which is the Laplace-Beltrami operator of $y_\Gamma$ on $\Gamma$. \\

Dynamic boundary conditions, also known as generalized Wentzell boundary conditions have been considered for many equations in mathematical physics and are motivated by various models such as chemical engineering, population dynamics, special flows in hydrodynamics, and so on. For more details and the physical meaning of this kind of boundary conditions, we refer the readers to \cite{cocang2008,gal2015,Golde}. \\

System \eqref{1.1} (with $u\equiv0$) combines the diffusion and convection in the bulk and surface equations. It describes many diffusion phenomena such as thermal processes with a particular speed. Such systems subject to stochastic disturbances also take into account all the small independent changes during the heat process. In this paper, we are interested in the study of the null controllability of \eqref{1.1}, so here the role of the control force $u$ comes into play to guide the heat flow so that it is zero at the final time.\\

Throughout this paper, we denote by $C$ a positive constant depending only on $G$, $G_0$, $A$, and $A_\Gamma$ which may change from one place to another. We also define the following  Hilbert space
$$\mathbb{H}^1=\{(y,y_\Gamma)\in H^1(G)\times H^1(\Gamma)\,:\, y\vert_\Gamma = y_\Gamma\},$$
equipped with the usual inner product
$$\langle(y,y_\Gamma),(z,z_\Gamma)\rangle_{\mathbb{H}^1} = \langle y,z\rangle_{H^1(G)} + \langle y_\Gamma,z_\Gamma\rangle_{H^1(\Gamma)},$$
with $H^1$ are the standard Sobolev spaces endowed with the canonical inner product.

In Section \ref{sec2}, we show that \eqref{1.1} is well-posed i.e., for any $(y_T,y_{\Gamma,T})\in L^2_{\mathcal{F}_T}(\Omega;\mathbb{L}^2)$ and $u \in L^2_\mathcal{F}(0,T;L^2(G_0))$, the equation \eqref{1.1} has a unique weak solution
\begin{align*}
(y,y_\Gamma,Y,\widetilde{Y})\in L^2_\mathcal{F}(0,T;\mathbb{H}^1)\times L^2_\mathcal{F}(0,T;\mathbb{L}^2).
\end{align*}
Moreover, there exists a constant $\mathcal{C}>0$ such that
\begin{align*}
&\,\vert(y,y_\Gamma)\vert_{L^2_\mathcal{F}(0,T;\mathbb{H}^1)}+\vert (Y,\widetilde{Y}) \vert_{L^2_\mathcal{F}(0,T;\mathbb{L}^2)}\\
&\leq \mathcal{C} \big(\vert(y_T,y_{\Gamma,T})\vert_{L^2_{\mathcal{F}_T}(\Omega;\mathbb{L}^2)} + \vert u\vert_{L^2_\mathcal{F}(0,T;L^2(G_0))}\big).
\end{align*}

The main result of this paper is the following null controllability of \eqref{1.1}.
\begin{thm}\label{thm1.1}
For any given $T>0$, $G_0\Subset G$ a nonempty open subset of $G$ and for all $(y_T,y_{\Gamma,T})\in L^2_{\mathcal{F}_T}(\Omega;\mathbb{L}^2)$, there exists a control $\hat{u}\in L^2_\mathcal{F}(0,T;L^2(G_0))$ such that the corresponding solution $(\hat{y},\hat{y}_\Gamma,\hat{Y},\hat{\widetilde{Y}})$ of system \eqref{1.1} satisfies
$$(\hat{y}(0,\cdot), \hat{y}_\Gamma(0,\cdot)) = (0,0)\,\,\,\,\,\mathbb{P}\textnormal{-a.s}.$$
Moreover, the control $\hat{u}$ can be chosen such that
\begin{align}\label{1.21203}
\vert \hat{u}\vert^2_{L^2_\mathcal{F}(0,T;L^2(G))}\leq e^{CK}\vert(y_T,y_{\Gamma,T})\vert^2_{L^2_{\mathcal{F}_T}(\Omega;\mathbb{L}^2)},
\end{align}
where $K$ has the following form
$$K\equiv1+\frac{1}{T}+\vert a_1\vert_\infty^{2/3}+T\vert a_1\vert_\infty+\vert b_1\vert_\infty^{2/3}+T\vert b_1\vert_\infty+\big(1+T\big)\big(\vert a_2\vert_\infty^2+\vert B\vert_\infty^2+\vert b_2\vert_\infty^2+\vert B_\Gamma\vert_\infty^2\big).$$
\end{thm}
\begin{rmk}
The null controllability of \eqref{1.1} with $B\equiv B_\Gamma\equiv 0$ and $A\equiv A_\Gamma\equiv I_N$ is already established in \cite[Theorem 5.1]{elgrouDBC}.
\end{rmk}

By the duality method, the null controllability of \eqref{1.1} can be reduced to the observability inequality of the following forward stochastic parabolic equation:
\begin{equation}\label{1.012}
{\small\begin{cases}
			\begin{array}{ll}
				dz - \text{div}(A\nabla z) \,dt = (-a_1 z+\text{div}(zB)) \,dt - a_2z \,dW(t) &\textnormal{in}\,\,Q,\\
				dz_\Gamma-\text{div}_\Gamma(A_\Gamma \nabla_\Gamma z_\Gamma) \,dt+\partial^A_\nu z \,dt = (-b_1z_\Gamma-z_\Gamma B\cdot \nu+\text{div}_\Gamma(z_\Gamma B_\Gamma))\,dt\\
    \hspace{5cm}\;\; -b_2z_\Gamma \,dW(t) &\textnormal{on}\,\,\Sigma,\\
				z_\Gamma(t,x)=z\vert_\Gamma(t,x) &\textnormal{on}\,\,\Sigma,\\
				(z,z_\Gamma)\vert_{t=0}=(z_0,z_{\Gamma,0}) &\textnormal{in}\,\,G\times\Gamma,
			\end{array}
		\end{cases}}
\end{equation}
where $(z_0,z_{\Gamma,0})\in L^2_{\mathcal{F}_0}(\Omega;\mathbb{L}^2)$ is the initial state.

In Section \ref{sec2}, we prove that \eqref{1.012} is well-posed i.e., for any $(z_0,z_{\Gamma,0})\in L^2_{\mathcal{F}_0}(\Omega;\mathbb{L}^2)$, the equation \eqref{1.012} has a unique weak solution $(z,z_\Gamma)\in L^2_\mathcal{F}(0,T;\mathbb{H}^1)$. Furthermore,
\begin{equation*}
\vert(z,z_\Gamma)\vert_{L^2_\mathcal{F}(0,T;\mathbb{H}^1)}\leq \mathcal{C}\, \vert(z_0,z_{\Gamma,0})\vert_{L^2_{\mathcal{F}_0}(\Omega;\mathbb{L}^2)}.
\end{equation*}

To establish the null controllability result of system \eqref{1.1}, we will prove the following appropriate observability inequality of \eqref{1.012}.
\begin{thm}\label{thm1.2}
For all $(z_0,z_{\Gamma,0})\in L^2_{\mathcal{F}_0}(\Omega;\mathbb{L}^2)$, the corresponding solution of \eqref{1.012} satisfies that
\begin{equation}\label{1.3}
    \vert z(T,\cdot)\vert^2_{L^2_{\mathcal{F}_T}(\Omega;L^2(G))}+\vert z_\Gamma(T,\cdot)\vert^2_{L^2_{\mathcal{F}_T}(\Omega;L^2(\Gamma))} \leq e^{CK}\,\mathbb{E}\int_{Q_0} z^2\,dxdt,
\end{equation}
where $K$ is as in Theorem \ref{thm1.1} and $C$ is the same positive constant appearing in \eqref{1.21203}.
\end{thm}
\begin{rmk}
Note that in Theorem \ref{thm1.2}, we give an explicit form of the observability constant appearing in \eqref{1.3} which has the form $Ce^{CT^{-1}}$ when $T$ is small.
\end{rmk}
\begin{rmk}
In this paper, we prove the observability inequality \eqref{1.3} by developing a new global Carleman estimate for equation \eqref{1.012}. Moreover, applying such a Carleman estimate, it is easy to show the unique continuation property for solutions of \eqref{1.012} i.e.,
$$z=0\,\,\textnormal{in}\,\,Q_0,\,\,\,\mathbb{P}\textnormal{-a.s}.\;\;\Longrightarrow\;\; (z_0,z_{\Gamma,0})=(0,0)\,\,\,\,\,\mathbb{P}\textnormal{-a.s.}$$
As well known, the above property can be used to establish the approximate controllability result for equation \eqref{1.1}, which is formulated as follows: For any final state $(y_T,y_{\Gamma,T})\in L^2_{\mathcal{F}_T}(\Omega;\mathbb{L}^2)$ and any desired state $(y_d,y_{\Gamma,d})\in L^2_{\mathcal{F}_0}(\Omega;\mathbb{L}^2)$, and all $\varepsilon>0$, there exists a control function $u\in L^2_\mathcal{F}(0,T;L^2(G_0))$ such that the corresponding solution $(y,y_\Gamma,Y,\widetilde{Y})$ to \eqref{1.1} satisfies $$\mathbb{E}\vert (y(0,\cdot),y_\Gamma(0,\cdot))-(y_d,y_{\Gamma,d})\vert^2_{\mathbb{L}^2}\leq\varepsilon.$$
\end{rmk}
In this paper, we will need the following special case of Itô's formulas when we want to show an energy estimate for solutions of equation \eqref{1.012} and for the duality relation between solutions of \eqref{1.1} and \eqref{1.012}. For more details about Itô's formulas, see, e.g., \cite[Chapter 2]{lu2021mathematical}.
\begin{lm}\label{lm1.1}
    Let $X,Y\in L^2_\mathcal{F}(0,T;H^1(G))$, $X_0\in L^2_{\mathcal{F}_0}(\Omega;L^2(G))$, $Y_T\in L^2_{\mathcal{F}_T}(\Omega;L^2(G))$, \\$\Phi,\Tilde{\Phi}\in L^2_\mathcal{F}(0,T;(H^1(G))')$ and $\Psi,Z\in L^2_\mathcal{F}(0,T;L^2(G))$, such that for all $t\in [0,T]$, the processes $X$ and $(Y,Z)$ satisfying respectively 
    $$X(t)=X_0+\int_0^t \Phi(s)ds+\int_0^t \Psi(s)dW(s),\,\,\;\mathbb{P}\textnormal{-a.s.},$$
    and
    $$Y(t)=Y_T-\int_t^T \Tilde{\Phi}(s)ds-\int_t^T Z(s)dW(s),\,\,\;\mathbb{P}\textnormal{-a.s.}$$
Then, we have
\begin{enumerate}
\item For all $t\in[0,T]$,
\begin{align*}
\vert X(t)\vert^2_{L^2(G)}=&\,\vert X_0\vert^2_{L^2(G)}+2\int_0^t \langle\Phi(s),X(s)\rangle_{(H^1(G))',H^1(G)} \,ds\\
&+2\int_0^t (\Psi(s),X(s))_{L^2(G)} \,dW(s)+\int_0^t \vert\Psi(s)\vert^2_{L^2(G)} \,ds,\,\,\mathbb{P}\textnormal{-a.s.}
\end{align*}
\item For all $t\in[0,T]$,
\begin{align*}
( X(t),Y(t))_{L^2(G)}=&\,( X_0,Y(0))_{L^2(G)}+\int_0^t \langle\Phi(s),Y(s)\rangle_{(H^1(G))',H^1(G)}ds\\
&+\int_0^t \langle\Tilde{\Phi}(s),X(s)\rangle_{(H^1(G))',H^1(G)}ds+\int_0^t (X(s),Z(s))_{L^2(G)}dW(s)\\
&+\int_0^t (Y(s),\Psi(s))_{L^2(G)}dW(s)+\int_0^t (\Psi(s),Z(s))_{L^2(G)}ds,\,\,\mathbb{P}\textnormal{-a.s.}
\end{align*}
\end{enumerate}
\end{lm}

In the above Lemma, $(\cdot,\cdot)_{L^2(G)}$ and $\langle\cdot,\cdot\rangle_{(H^{1}(G))',H^1(G)}$ denote respectively the inner product in $L^2(G)$ and the duality product between $H^1(G)$ and its dual space $(H^{1}(G))'$ w.r.t to the pivot space $L^2(G)$.
\begin{rmk}
We can also give a similar Itô's formulas as in Lemma \ref{lm1.1} for some stochastic processes $X_\Gamma,Y_\Gamma\in L^2_\mathcal{F}(0,T;H^1(\Gamma))$. In general, for simplicity of notations, we use the differential form of Itô's formula which is given as follows $``d(X,Y)=(dX,Y)+(X,dY)+(dX,dY)"$ where $(\cdot,\cdot)$ designed either the inner product of the space $L^2(G)$ or $L^2(\Gamma)$. 
\end{rmk}
Note that the operators $\operatorname{div}$ and $\operatorname{div}_\Gamma$ appearing in the adjoint equation \eqref{1.012} are to be understood in the weak sense. Dealing with this, we use the following result. For the proof, see, e.g., \cite{DesZol,khoutaibi2020null}.
\begin{lm}\label{lm1.2}
For any $F\in L^2(G;\mathbb{R}^N)$ and $F_\Gamma\in L^2(\Gamma;\mathbb{R}^N)$, we define the extension of the operator divergence $($resp., tangential divergence$)$ of $F$ $($resp., $F_\Gamma$$)$ as the linear continuous operator on $H^1(G)$ $($resp., on $H^1(\Gamma)$$)$ as follows
\begin{equation*}
{\normalfont\text{div}}(F):H^1(G)\longrightarrow\mathbb{R},\,\, v \longmapsto -\int_G F\cdot\nabla v\, dx + \langle F\cdot\nu,v\vert_\Gamma\rangle_{H^{-1/2}(\Gamma),H^{1/2}(\Gamma)},
\end{equation*}
\begin{equation*}
{\normalfont\text{div}}_\Gamma(F_\Gamma):H^1(\Gamma)\longrightarrow\mathbb{R},\,\, v_\Gamma \longmapsto -\int_\Gamma F_\Gamma\cdot\nabla_\Gamma v_\Gamma\, d\sigma.
\end{equation*}
\end{lm}
\begin{rmk}
Here, without saying it the notion of vectors on $\Gamma$ must be interpreted in the sense of vector field on the variety $\Gamma$, that is to say as a function $\gamma:\Gamma\rightarrow\mathbb{R}^N$ so that $\gamma(x)\in T_x\Gamma$, where
$$T_x\Gamma=\{p\in\mathbb{R}^N,\quad p\cdot\nu(x)=0\}$$
is the tangent space at the point $x\in\Gamma$.
\end{rmk}

This paper is organized as follows. In Section \ref{sec2}, we study the well-posedness of equations \eqref{1.1} and \eqref{1.012}. In Section \ref{sec3}, we establish a global Carleman estimate for general forward stochastic parabolic equations and deduce the proof of Theorem \ref{thm1.2}. In Section \ref{sec4}, we prove Theorem \ref{thm1.1}.
\section{Well-posedness of equations}\label{sec2}
This section is devoted to studying the well-posedness of the forward (resp., backward) linear stochastic parabolic equation \eqref{1.012} (resp., \eqref{1.1}). For more details on the well-posedness and regularity of solutions of stochastic equations, we refer the readers to \cite{da2014stochastic,pardoux1990adapted,X. Zhou}.

We first define the weak solution of the system \eqref{1.012}.
\begin{df}
The process $(z,z_\Gamma)\in L^2_\mathcal{F}(0,T;\mathbb{H}^1)$ is said to be a weak solution of \eqref{1.012} if for any $(\eta,\eta_\Gamma)\in\mathbb{H}^1$ and all $t\in[0,T]$, it holds that
\begin{align*}
&\,\int_G (z(t)-z_0)\eta \,dx + \int_\Gamma(z_\Gamma(t)-z_{\Gamma,0})\eta_\Gamma \,d\sigma\\
&=-\int_0^t\int_G A\nabla z\cdot\nabla\eta \,dx ds-\int_0^t\int_G a_1z\eta \,dx ds-\int_0^t\int_G zB\cdot\nabla\eta \,dxds\\
&\quad\,-\int_0^t\int_G a_2z\eta \,dx dW(s)-\int_0^t\int_\Gamma A_\Gamma\nabla_\Gamma z_\Gamma\cdot\nabla_\Gamma \eta_\Gamma \,d\sigma ds-\int_0^t\int_\Gamma b_1z_\Gamma\eta_\Gamma \,d\sigma ds\\
&\quad\,-\int_0^t\int_\Gamma z_\Gamma B_\Gamma \cdot\nabla_\Gamma\eta_\Gamma \,d\sigma ds-\int_0^t\int_\Gamma b_2z_\Gamma\eta_\Gamma \,d\sigma dW(s),\quad \mathbb{P}\textnormal{-a.s.}
\end{align*}
\end{df}
We have the following well-posedness result of system \eqref{1.012}.
\begin{thm}\label{thm2.1}
For any initial state $(z_0,z_{\Gamma,0})\in L^2_{\mathcal{F}_0}(\Omega;\mathbb{L}^2)$, there exists a unique weak solution $(z,z_\Gamma)\in L^2_\mathcal{F}(0,T;\mathbb{H}^1)$ of \eqref{1.012}. Moreover,
\begin{align}\label{deppconine}
\vert(z,z_\Gamma)\vert_{L^2_\mathcal{F}(0,T;\mathbb{H}^1)}\leq C\, \vert(z_0,z_{\Gamma,0})\vert_{L^2_{\mathcal{F}_0}(\Omega;\mathbb{L}^2)}.
\end{align}
\end{thm}
\begin{proof}  
For the uniqueness, let us suppose that $(z,z_\Gamma)\in L^2_\mathcal{F}(0,T;\mathbb{H}^1)$ is a weak solution of \eqref{1.012} with initial condition $(z_0,z_{\Gamma,0})=(0,0)$. Fix $t\in[0,T]$ and by Itô's formula in \ref{lm1.1}, we obtain that
\begin{align*}
&\,\mathbb{E}\int_G z^2(t)\,dx+\mathbb{E}\int_\Gamma  z^2_\Gamma(t)\,d\sigma\\
&=-2\mathbb{E}\int_0^t\int_G A\nabla z\cdot\nabla z dxds-2\mathbb{E}\int_0^t\int_G a_1z^2dxds-2\mathbb{E}\int_0^t\int_G zB\cdot\nabla z dxds\\
&\quad\,+\mathbb{E}\int_0^t\int_G a_2^2z^2dxds-2\mathbb{E}\int_0^t\int_\Gamma A_\Gamma\nabla_\Gamma z_\Gamma\cdot\nabla_\Gamma z_\Gamma d\sigma ds-2\mathbb{E}\int_0^t\int_\Gamma b_1z_\Gamma^2d\sigma ds\\
&\quad\,-2\mathbb{E}\int_0^t\int_\Gamma z_\Gamma B_\Gamma \cdot\nabla_\Gamma z_\Gamma d\sigma ds+\mathbb{E}\int_0^t\int_\Gamma b_2^2z_\Gamma^2 d\sigma ds.
\end{align*}
Utilizing assumptions on $A$, $A_\Gamma$ and applying Young's inequality, it follows that for any $\varepsilon>0$
\begin{align*}
&\,\mathbb{E}\int_G  z^2(t)\,dx+\mathbb{E}\int_\Gamma  z^2_\Gamma(t)\,d\sigma\\
&\leq -2\beta\mathbb{E}\int_0^t\int_G \vert\nabla z\vert^2dxds+2\vert a_1\vert_\infty \mathbb{E}\int_0^t\int_G z^2 dxds+\varepsilon\mathbb{E}\int_0^t\int_G \vert\nabla z\vert^2dxds\\
&\quad\,+\vert B\vert^2_\infty C(\varepsilon)\mathbb{E}\int_0^t\int_G z^2 dxds+\vert a_2\vert_\infty^2\mathbb{E}\int_0^t\int_G z^2 dxds-2\beta\mathbb{E}\int_0^t\int_\Gamma \vert\nabla_\Gamma z_\Gamma\vert^2d\sigma ds\\
&\quad\,+2\vert b_1\vert_\infty\mathbb{E}\int_0^t\int_\Gamma z_\Gamma^2d\sigma ds+\varepsilon\mathbb{E}\int_0^t\int_\Gamma \vert\nabla_\Gamma z_\Gamma\vert^2 d\sigma ds+\vert B_\Gamma\vert^2_\infty C(\varepsilon)\mathbb{E}\int_0^t\int_\Gamma z_\Gamma^2 d\sigma ds\\
&\quad\,+\vert b_2\vert_\infty^2\mathbb{E}\int_0^t\int_\Gamma z_\Gamma^2 d\sigma ds.
\end{align*}
Now, by choosing a small $0<\varepsilon\leq2\beta$, we conclude that
$$\mathbb{E}\int_G  z^2(t)\,dx+\mathbb{E}\int_\Gamma  z^2_\Gamma(t)\,d\sigma\leq \mathcal{C}\,\,\Bigg(\mathbb{E}\int_0^t\int_G z^2dxds+\mathbb{E}\int_0^t\int_\Gamma z_\Gamma^2d\sigma ds\Bigg),$$
where $\mathcal{C}$ is a positive constant depending on $\vert a_1\vert_\infty$, $\vert a_2\vert_\infty$, $\vert B\vert_\infty$, $\vert b_1\vert_\infty$, $\vert b_2\vert_\infty$, and $\vert B_\Gamma\vert_\infty$. Hence, by Gronwall's inequality, we obtain that $(z(t),z_\Gamma(t))=(0,0)$. This concludes the uniqueness of solutions of system \eqref{1.012}.

For the existence of solutions, we borrow some ideas from \cite{krylov,X. Zhou}. We first consider an appropriate approximation of \eqref{1.012}, and then we show that the sequence solution of such approximation converges to the weak solution of \eqref{1.012}. For simplicity of notations, we denote by $\langle\cdot,\cdot\rangle_1$ (resp., $\langle\cdot,\cdot\rangle_{1,\Gamma}$) the duality pairing between $(H^1(G))'$ and $H^1(G)$ (resp., $H^{-1}(\Gamma)$ and $H^1(\Gamma)$), under $(L^2(G))'\equiv L^2(G)$ (resp., $(L^2(\Gamma))'\equiv L^2(\Gamma))$. We adopt also the notation $(\cdot,\cdot)_0$ (resp., $(\cdot,\cdot)_{0,\Gamma}$) for the inner product in $L^2(G)$ (resp., $L^2(\Gamma)$). 

Let us now rewrite equation \eqref{1.012} under the following simple form
\begin{equation}\label{2.7}
		\begin{cases}
  \begin{array}{ll}
 dz(t)=(\mathcal{A}(t)z(t)-a_1(t)z(t))dt-a_2(t)z(t)dW(t)&\textnormal{in}\,\,Q,\\
   dz_\Gamma(t)=(\mathcal{A}_\Gamma(t) z_\Gamma(t)-b_1(t)z_\Gamma(t))dt-b_2(t)z_\Gamma(t)dW(t)&\textnormal{on}\,\,\Sigma,\\
   z_\Gamma(t,\cdot)=z\vert_\Gamma(t,\cdot)&\textnormal{on}\,\,\Sigma,\\
   (z,z_\Gamma)\vert_{t=0}=(z_0,z_{\Gamma,0})&\textnormal{in}\,\,G\times\Gamma,
   \end{array}
	\end{cases}\end{equation}
where
\begin{align*}
&\,\mathcal{A}(t)z(t)=\textnormal{div}(A\nabla z)+\textnormal{div}(zB),\\
&\mathcal{A}_\Gamma(t) z_\Gamma(t)=\textnormal{div}_\Gamma(A_\Gamma\nabla_\Gamma z_\Gamma)-\partial^A_\nu z-z_\Gamma B\cdot\nu+\textnormal{div}_\Gamma(z_\Gamma B_\Gamma ).
\end{align*}

Let $(e_i,e_{i,\Gamma})_{i\geq1}$ be a Hilbert basis of $\mathbb{H}^1$, which is orthonormal as a basis of $\mathbb{L}^2$. Fix an integer $n\geq1$, then for all $i=1,2,...,n$, there exists $z_{ni}\in L^2_\mathcal{F}(0,T;\mathbb{R})$ satisfying the following stochastic differential equation: 
\begin{align}\label{2.8}
\begin{cases}
&dz_{ni}(t)=\displaystyle\sum_{j=1}^n \bigg[\langle\mathcal{A}(t)e_j,e_i\rangle_1 -(a_1(t)e_j,e_i)_0+\langle\mathcal{A}_\Gamma(t) e_{j,\Gamma},e_{i,\Gamma}\rangle_{1,\Gamma}\\
&\hspace{2.2cm}-(b_1(t)e_{j,\Gamma},e_{i,\Gamma})_{0,\Gamma}\bigg]z_{nj}(t) \,dt+\displaystyle\sum_{j=1}^n \bigg[-(a_2(t)e_j,e_i)_0\\
&\hspace{2.2cm}-(b_2(t)e_{j,\Gamma},e_{i,\Gamma})_{0,\Gamma}\bigg] z_{nj}(t) \,dW(t),\qquad\quad t\in(0,T),\\
&z_{ni}(0)=z_{ni,0},
\end{cases}
\end{align}
where $z_{ni,0}=\langle(z_0,z_{\Gamma,0}),(e_i,e_{i,\Gamma})\rangle_{\mathbb{L}^2}$. Let us define $$(z_n(t),z_{n,\Gamma}(t)):=\displaystyle\sum_{i=1}^n z_{ni}(t)(e_i,e_{i,\Gamma})\in L^2_\mathcal{F}(0,T;\mathbb{H}^1).$$ 
Note also that
\begin{equation}\label{2.30}
(z_{n,0},z_{n,0\Gamma})=(z_n(0),z_{n,\Gamma}(0)) \longrightarrow (z_0,z_{\Gamma,0})\,\,\,\textnormal{in}\,\,L^2_{\mathcal{F}_0}(\Omega;\mathbb{L}^2)\,\,\,\textnormal{as}\,\,n\rightarrow\infty.
\end{equation}
Now, by Itô's formula, we deduce that for all $t\in[0,T]$
\begin{align}\label{2.4159}
\begin{aligned}
&\,\mathbb{E}\vert z_n(t)\vert_{L^2(G)}^2-\mathbb{E}\vert z_{n,0}\vert_{L^2(G)}^2+\mathbb{E}\vert z_{n,\Gamma}(t)\vert_{L^2(\Gamma)}^2-\mathbb{E}\vert z_{n,0\Gamma}\vert_{L^2(\Gamma)}^2\\
&\leq C\mathbb{E}\int_0^t\int_G \vert z_n\vert^2dxds-2\beta\mathbb{E}\int_0^t\int_G \vert\nabla z_n\vert^2dxds+C\mathbb{E}\int_0^t\int_G \vert z_n\vert\vert\nabla z_n\vert dxds\\
&\quad\,+C\mathbb{E}\int_0^t\int_\Gamma \vert z_{n,\Gamma}\vert^2d\sigma ds-2\beta\mathbb{E}\int_0^t\int_\Gamma \vert\nabla_\Gamma z_{n,\Gamma}\vert^2d\sigma ds+C\mathbb{E}\int_0^t\int_\Gamma \vert z_{n,\Gamma}\vert\vert\nabla_\Gamma z_{n,\Gamma}\vert d\sigma ds.
\end{aligned}
\end{align}
By Young's inequality, \eqref{2.4159} implies that for any $\varepsilon>0$
\begin{align*}
&\,\mathbb{E}\vert z_n(t)\vert_{L^2(G)}^2-\mathbb{E}\vert z_{n,0}\vert_{L^2(G)}^2+\mathbb{E}\vert z_{n,\Gamma}(t)\vert_{L^2(\Gamma)}^2-\mathbb{E}\vert z_{n,0\Gamma}\vert_{L^2(\Gamma)}^2\\
&\leq C\mathbb{E}\int_0^t\int_G \vert z_n\vert^2dxds-2\beta\mathbb{E}\int_0^t\int_G \vert\nabla z_n\vert^2dxds+\varepsilon\mathbb{E}\int_0^t\int_G \vert\nabla z_n\vert^2 dxds\\
&\quad\,+C(\varepsilon)\mathbb{E}\int_0^t\int_G \vert z_n\vert^2 dxds+C\mathbb{E}\int_0^t\int_\Gamma \vert z_{n,\Gamma}\vert^2d\sigma ds-2\beta\mathbb{E}\int_0^t\int_\Gamma \vert\nabla_\Gamma z_{n,\Gamma}\vert^2d\sigma ds\\
&\quad\,+\varepsilon\mathbb{E}\int_0^t\int_\Gamma \vert\nabla_\Gamma z_{n,\Gamma}\vert^2 d\sigma ds+C(\varepsilon)\mathbb{E}\int_0^t\int_\Gamma \vert z_{n,\Gamma}\vert^2 d\sigma ds.
\end{align*}
Taking a small enough $0<\varepsilon<2\beta$, we deduce that
\begin{align*}
&\,\mathbb{E}\vert z_n(t)\vert_{L^2(G)}^2+\mathbb{E}\vert z_{n,\Gamma}(t)\vert_{L^2(\Gamma)}^2+\mathbb{E}\int_0^t\int_G \vert\nabla z_{n}\vert^2 dx ds+\mathbb{E}\int_0^t\int_\Gamma \vert\nabla_\Gamma z_{n,\Gamma}\vert^2 d\sigma ds\\
&\leq C(\mathbb{E}\vert z_{n,0}\vert_{L^2(G)}^2+\mathbb{E}\vert z_{n,0\Gamma}\vert_{L^2(\Gamma)}^2)+C\mathbb{E}\int_0^t\int_G \vert z_{n}\vert^2dxds+C\mathbb{E}\int_0^t\int_\Gamma\vert z_{n,\Gamma}\vert^2d\sigma ds.
\end{align*}
Then, it follows that
\begin{align*}
&\,\mathbb{E}\vert(z_n(t),z_{n,\Gamma}(t))\vert^2_{\mathbb{L}^2}+\mathbb{E}\int_0^t \vert (z_{n}(s),z_{n,\Gamma}(s))\vert^2_{\mathbb{H}^1} ds\\
&\leq C\mathbb{E}\vert(z_{n,0},z_{n,0\Gamma})\vert^2_{\mathbb{L}^2}+C\mathbb{E}\int_0^t \vert (z_{n}(s),z_{n,\Gamma}(s))\vert^2_{\mathbb{L}^2} ds.
\end{align*}
Hence, by Gronwall's inequality, we end up with
\begin{align}\label{2.811}
\begin{aligned}
&\,\sup_{0\leq t\leq T}\mathbb{E}\vert(z_n(t),z_{n,\Gamma}(t))\vert^2_{\mathbb{L}^2}+ \mathbb{E}\int_0^T\vert (z_{n}(s),z_{n,\Gamma}(s))\vert^2_{\mathbb{H}^1} ds\\
&\leq C\,\mathbb{E}\vert (z_{n,0}, z_{n,0\Gamma})\vert_{\mathbb{L}^2}^2.
\end{aligned}
\end{align}
Recalling \eqref{2.30}, then from \eqref{2.811}, we conclude that there exist a subsequence of $(z_n,z_{n,\Gamma})$ (still denotes also by $(z_n,z_{n,\Gamma})$ for simplicity of notations) and  $(z,z_\Gamma)\in L^2_\mathcal{F}(0,T;\mathbb{H}^1)$ such that as $n\rightarrow\infty$
\begin{equation}\label{2.812}
(z_n,z_{n,\Gamma}) \longrightarrow (z,z_\Gamma),\quad\textnormal{weakly in}\quad L^2_\mathcal{F}(0,T;\mathbb{H}^1).
\end{equation}
We claim that $(z,z_\Gamma)$ is the weak solution of \eqref{1.012}. To prove this fact, let us consider $\gamma:[0,T]\rightarrow\mathbb{R}$ an absolutely continuous function with $\overset{.}{\gamma}(t):=d\gamma/dt\in L^2(0,T)$ and $\gamma(T)=0$. For all integer $i\geq1$, set $(\gamma_i(t),\gamma_{i,\Gamma}(t)):=\gamma(t)(e_i,e_{i,\Gamma})$. From \eqref{2.8} and by Itô's formula, we obtain that
\begin{align*}
&\,-(z_{n,0},\gamma_i(0))_0-(z_{n,0\Gamma},\gamma_{i,\Gamma}(0))_{0,\Gamma}\\
&=\int_0^T \Big[\langle\mathcal{A}(t)z_n(t),\gamma_i(t)\rangle_1-(a_1(t)z_n(t),\gamma_i(t))_0]dt-\int_0^T (a_2(t)z_n(t),\gamma_i(t))_0dW(t)\\
&\quad\,+\int_0^T (z_n(t),\overset{.}{\gamma_i}(t))_0dt+\int_0^T \Big[\langle\mathcal{A}_\Gamma(t) z_{n,\Gamma}(t),\gamma_{i,\Gamma}(t)\rangle_{1,\Gamma}-(b_1(t)z_{n,\Gamma}(t),\gamma_{i,\Gamma}(t))_{0,\Gamma}]dt\\
&\quad\,-\int_0^T (b_2(t)z_{n,\Gamma}(t),\gamma_{i,\Gamma}(t))_{0,\Gamma}dW(t)+\int_0^T (z_{n,\Gamma}(t),\overset{.}{\gamma}_{i,\Gamma}(t))_{0,\Gamma}dt,\qquad\textnormal{for all}\,\,i\geq1.
\end{align*}
It follows that for all $(\eta,\eta_\Gamma)\in\mathbb{H}^1$, we have
\begin{align*}
\begin{aligned}
&\,-(z_{n,0},\eta)_0\gamma(0)-(z_{n,0\Gamma},\eta_\Gamma)_{0,\Gamma}\gamma(0)\\
&=\int_0^T \Big[\langle\mathcal{A}(t)z_n(t),\eta\rangle_1-(a_1(t)z_n(t),\eta)_0]\gamma(t) dt-\int_0^T (a_2(t)z_n(t),\eta)_0\gamma(t) dW(t)\\
&\quad\,+\int_0^T (z_n(t),\eta)_0\overset{.}{\gamma}(t)dt+\int_0^T \Big[\langle\mathcal{A}_\Gamma(t) z_{n,\Gamma}(t),\eta_{\Gamma}\rangle_{1,\Gamma}-(b_1(t)z_{n,\Gamma}(t),\eta_{\Gamma})_{0,\Gamma}\Big]\gamma(t) dt\\
&
\quad\,-\int_0^T (b_2(t)z_{n,\Gamma}(t),\eta_{\Gamma})_{0,\Gamma}\gamma(t) dW(t)+\int_0^T (z_{n,\Gamma}(t),\eta_\Gamma)_{0,\Gamma}\overset{.}{\gamma}(t)dt\qquad \mathbb{P}\textnormal{-a.s.}
\end{aligned}
\end{align*}
Hence, by integration by parts, we arrive at
\begin{align}
\begin{aligned}
&\,-(z_{n,0},\eta)_0\gamma(0)-(z_{n,0\Gamma},\eta_\Gamma)_{0,\Gamma}\gamma(0)\\
&=\int_Q \Big[-A\nabla z_n\cdot\nabla\eta-z_nB\cdot\nabla\eta-a_1 z_n\eta]\gamma(t) \,dxdt-\int_Q a_2 z_n \eta\gamma(t) dx dW(t)\\
&\quad\,+\int_Q z_n \eta \overset{.}{\gamma}(t) \,dxdt+\int_\Sigma \Big[-A_\Gamma\nabla_\Gamma z_{n,\Gamma}\cdot\nabla_\Gamma\eta_{\Gamma}-z_{n,\Gamma}B_\Gamma\cdot\nabla_\Gamma\eta_\Gamma\\
&\label{2.1002}
\quad\,-b_1 z_{n,\Gamma}\eta_{\Gamma}\Big]\gamma(t) \,d\sigma dt-\int_\Sigma b_2 z_{n,\Gamma}\eta_{\Gamma}\gamma(t) d\sigma dW(t)+\int_\Sigma z_{n,\Gamma}\eta_\Gamma \overset{.}{\gamma}(t) \,d\sigma dt.
\end{aligned}
\end{align}
Letting $n\rightarrow\infty$ and using \eqref{2.30} and \eqref{2.812}, the equality \eqref{2.1002} provides that
\begin{align}
\begin{aligned}
&\,-(z_{0},\eta)_0\gamma(0)-(z_{\Gamma,0},\eta_\Gamma)_{0,\Gamma}\gamma(0)\\
&=\int_Q \Big[-A\nabla z\cdot\nabla\eta-zB\cdot\nabla\eta-a_1 z\eta]\gamma(t) \,dxdt-\int_Q a_2 z \eta\gamma(t) dx dW(t)\\
&\quad\,+\int_Q z \eta \overset{.}{\gamma}(t) \,dxdt+\int_\Sigma \Big[-A_\Gamma\nabla_\Gamma z_{\Gamma}\cdot\nabla_\Gamma\eta_{\Gamma}-z_{\Gamma}B_\Gamma\cdot\nabla_\Gamma\eta_\Gamma\\
&\label{2.1003}
\quad\,-b_1 z_{\Gamma}\eta_{\Gamma}\Big]\gamma(t) \,d\sigma dt-\int_\Sigma b_2 z_{\Gamma}\eta_{\Gamma}\gamma(t) d\sigma dW(t)+\int_\Sigma z_{\Gamma}\eta_\Gamma \overset{.}{\gamma}(t) \,d\sigma dt.
\end{aligned}
\end{align}
For any $t\in(0,T)$ and $\varepsilon>0$, we define $\gamma_\varepsilon$ by
$$
\gamma_\varepsilon(s):=
\begin{cases}
\begin{array}{llll}
1&&\textnormal{if}&0\leq s\leq t-\varepsilon/2,\\
-\frac{1}{\varepsilon}(s-t-\varepsilon/2)&&\textnormal{if}&t-\varepsilon/2<s< t+\varepsilon/2,\\
0&&\textnormal{if}&t+\varepsilon/2\leq s\leq T.
\end{array}
\end{cases}
$$
Using $\gamma_\varepsilon$ in \eqref{2.1003} and letting $\varepsilon\rightarrow0$, we obtain that
\begin{align*}
&\,\int_G (z(t)-z_0)\eta \,dx + \int_\Gamma(z_\Gamma(t)-z_{\Gamma,0})\eta_\Gamma \,d\sigma\\
&=-\int_0^t\int_G A\nabla z\cdot\nabla\eta \,dx ds-\int_0^t\int_G a_1z\eta \,dx ds-\int_0^t\int_G zB\cdot\nabla\eta \,dxds\\
&\quad\,-\int_0^t\int_G a_2z\eta \,dx dW(s)-\int_0^t\int_\Gamma A_\Gamma\nabla_\Gamma z_\Gamma\cdot\nabla_\Gamma \eta_\Gamma \,d\sigma ds-\int_0^t\int_\Gamma b_1z_\Gamma\eta_\Gamma \,d\sigma ds\\
&\quad\,-\int_0^t\int_\Gamma z_\Gamma B_\Gamma \cdot\nabla_\Gamma\eta_\Gamma \,d\sigma ds-\int_0^t\int_\Gamma b_2z_\Gamma\eta_\Gamma \,d\sigma dW(s)\qquad \mathbb{P}\textnormal{-a.s.}
\end{align*}
Hence, we deduce that $(z,z_\Gamma)\in L^2_\mathcal{F}(0,T;\mathbb{H}^1)$ is the weak solution of \eqref{1.012}. Moreover, the desired inequality \eqref{deppconine} follows immediately from \eqref{2.811} and \eqref{2.812}. This completes the proof of Theorem \ref{thm2.1}.
\end{proof}
Now, let us define the weak solution of the following backward stochastic parabolic equation
\begin{equation}\label{2.70}
{\small\begin{cases}
			\begin{array}{ll}
				dy + \textnormal{div}(A\nabla y) \,dt = (a_1 y+a_2 Y +B\cdot\nabla y +f) \,dt + Y \,dW(t) &\textnormal{in}\,Q,\\
			 
 dy_\Gamma+\textnormal{div}_\Gamma(A_\Gamma\nabla_\Gamma y_\Gamma) \,dt-\partial_\nu^A y \,dt = (b_1y_\Gamma+b_2\widetilde{Y}+B_\Gamma\cdot\nabla_\Gamma y_\Gamma)\,dt+\widetilde{Y} \,dW(t) &\textnormal{on}\,\Sigma,\\
				y_\Gamma(t,x)=y\vert_\Gamma(t,x) &\textnormal{on}\,\Sigma,\\
				(y,y_\Gamma)\vert_{t=T}=(y_T,y_{\Gamma,T}) &\textnormal{in}\,G\times\Gamma,
			\end{array}
		\end{cases}}
\end{equation}
where $a_1, a_2\in L_\mathcal{F}^\infty(0,T;L^\infty(G))$, $B\in L_\mathcal{F}^\infty(0,T;L^\infty(G;\mathbb{R}^N))$, $b_1, b_2\in L_\mathcal{F}^\infty(0,T;L^\infty(\Gamma))$, $B_\Gamma\in L_\mathcal{F}^\infty(0,T;L^\infty(\Gamma;\mathbb{R}^N))$, $(y_T,y_{\Gamma,T})\in L^2_{\mathcal{F}_T}(\Omega;\mathbb{L}^2)$ and $f \in L^2_\mathcal{F}(0,T;L^2(G))$.
\begin{df}
The process 
$$(y,y_\Gamma,Y,\widetilde{Y})\in  L^2_\mathcal{F}(0,T;\mathbb{H}^1)\times L^2_\mathcal{F}(0,T;\mathbb{L}^2)$$ is said to be a weak solution of equation \eqref{2.70} if for any $(\eta,\eta_\Gamma)\in\mathbb{H}^1$ and all $t\in[0,T]$, it holds that
\begin{align*}
&\,\int_G (y_T-y(t))\eta \,dx + \int_\Gamma(y_{\Gamma,T}-y_{\Gamma}(t))\eta_\Gamma \,d\sigma\\
&=\int_t^T\int_G A\nabla y\cdot\nabla\eta \,dx ds+\int_t^T\int_G a_1y\eta \,dx ds+\int_t^T\int_G a_2Y\eta \,dx ds\\
&\quad\,+\int_t^T\int_G \eta B\cdot\nabla y \,dxds+\int_t^T\int_G f \eta \,dx ds+\int_t^T\int_G Y \eta \,dxdW(s)\\
&\quad\,+\int_t^T\int_\Gamma A_\Gamma\nabla_\Gamma y_\Gamma\cdot\nabla_\Gamma \eta_\Gamma \,d\sigma ds+\int_t^T\int_\Gamma b_1y_\Gamma\eta_\Gamma \,d\sigma ds+\int_t^T\int_\Gamma b_2\widetilde{Y}\eta_\Gamma \,d\sigma ds\\
&\quad\,+\int_t^T\int_\Gamma \eta_\Gamma B_\Gamma\cdot \nabla_\Gamma y_\Gamma \,d\sigma ds+\int_t^T\int_\Gamma \widetilde{Y}\eta_\Gamma \,d\sigma dW(s).
\end{align*}
\end{df}
Repeating some ideas from the proof of Theorem \ref{thm2.1}, and adapting the duality approach used in \cite{X. Zhou}, we obtain the following well-posedness result.
\begin{thm}
Let $a_1, a_2\in L_\mathcal{F}^\infty(0,T;L^\infty(G))$, $B\in L_\mathcal{F}^\infty(0,T;L^\infty(G;\mathbb{R}^N))$, $b_1, b_2\in L_\mathcal{F}^\infty(0,T;L^\infty(\Gamma))$, $B_\Gamma\in L_\mathcal{F}^\infty(0,T;L^\infty(\Gamma;\mathbb{R}^N))$, $f\in L_\mathcal{F}^2(0,T;L^2(G))$ and $(y_T,y_{\Gamma,T})\in L^2_{\mathcal{F}_T}(\Omega;\mathbb{L}^2)$. Then equation \eqref{2.70} $($and hence \eqref{1.1}$)$ admits a unique weak solution 
$$(y,y_\Gamma,Y,\widetilde{Y})\in L^2_\mathcal{F}(0,T;\mathbb{H}^1)\times L^2_\mathcal{F}(0,T;\mathbb{L}^2).$$
Furthermore,
\begin{align*}
&\,\vert(y,y_\Gamma)\vert_{L^2_\mathcal{F}(0,T;\mathbb{H}^1)}+\vert (Y,\widetilde{Y}) \vert_{L^2_\mathcal{F}(0,T;\mathbb{L}^2)} \\
&\leq C \big(\vert(y_T,y_{\Gamma,T})\vert_{L^2_{\mathcal{F}_T}(\Omega;\mathbb{L}^2)}+\vert f\vert_{L^2_\mathcal{F}(0,T;L^2(G))}\big).
 \end{align*}
\end{thm}
\section{Global Carleman estimate for forward stochastic parabolic equations and proof of Theorem \ref{thm1.2}}\label{sec3}
This section is devoted to establishing a global Carleman estimate for forward stochastic parabolic equations with weak divergence source terms and dynamic boundary conditions. Then, we deduce the proof of Theorem \ref{thm1.2}.
\subsection{Global Carleman estimate}
In this subsection, we establish a Carleman estimate for the following in-homogeneous  forward stochastic parabolic equation:
	\begin{equation}\label{3.1}
		{\small\begin{cases}
			\begin{array}{ll}
				dz - \textnormal{div}(A\nabla z) \,dt = (F_0+\textnormal{div}(F)) \,dt + F_1 \,dW(t) &\textnormal{in}\,Q,\\
				dz_\Gamma-\textnormal{div}_\Gamma(A_\Gamma\nabla_\Gamma z_\Gamma) \,dt+\partial^A_\nu z \,dt = (F_{0,\Gamma}-F\cdot\nu+\textnormal{div}_\Gamma(F_\Gamma)) \,dt\\
    \hspace{5cm}\;+ F_{1,\Gamma} \,dW(t) &\textnormal{on}\,\Sigma,\\
				z_\Gamma(t,x)=z\vert_\Gamma(t,x) &\textnormal{on}\,\Sigma,\\
				(z,z_\Gamma)\vert_{t=0}=(z_0,z_{\Gamma,0}) &\textnormal{in}\,G\times\Gamma,
			\end{array}
		\end{cases}}
	\end{equation}
	where $(z_0,z_{\Gamma,0})\in L^2_{\mathcal{F}_0}(\Omega;\mathbb{L}^2)$ is the initial state and the coefficients $F_0, F_1\in L^2_\mathcal{F}(0,T;L^2(G))$, $F\in L^2_\mathcal{F}(0,T;L^2(G,\mathbb{R}^N))$, $F_{0,\Gamma}, F_{1,\Gamma}\in L^2_\mathcal{F}(0,T;L^2(\Gamma))$, $F_\Gamma\in L^2_\mathcal{F}(0,T;L^2(\Gamma;\mathbb{R}^N))$. Similarly to the proof of Theorem \ref{thm2.1}, one can show also that the system \eqref{3.1} is well-posed i.e., there exists a unique weak solution
 $$(z,z_\Gamma)\in L^2_\mathcal{F}(0,T;\mathbb{H}^1)$$ of equation \eqref{3.1} so that
\begin{align*}
\vert(z,z_\Gamma)\vert_{L^2_\mathcal{F}(0,T;\mathbb{H}^1)}&\leq C\, \big(\vert(z_0,z_{\Gamma,0})\vert_{L^2_{\mathcal{F}_0}(\Omega;\mathbb{L}^2)}+\vert F_0\vert_{L^2_\mathcal{F}(0,T;L^2(G))}+\vert F_1\vert_{L^2_\mathcal{F}(0,T;L^2(G))}\\
&\quad\,\quad\;\,+\vert F\vert_{L^2_\mathcal{F}(0,T;L^2(G;\mathbb{R}^N))}+\vert F_{0,\Gamma}\vert_{L^2_\mathcal{F}(0,T;L^2(\Gamma))}\\
&\quad\,\quad\;\,+\vert F_{1,\Gamma}\vert_{L^2_\mathcal{F}(0,T;L^2(\Gamma))}+\vert F_\Gamma\vert_{L^2_\mathcal{F}(0,T;L^2(\Gamma;\mathbb{R}^N))}\big).
\end{align*}

We will need the following technical lemma due to Imanuvilov. For the proof, we refer to \cite{fursikov1996controllability}.
\begin{lm}\label{lm3.1}
		For any nonempty open subset $G_1\Subset G$, there is a function $\psi\in C^2(\overline{G})$ such that
		$$
		\psi>0,\,\, \textnormal{in} \,\,G\,;\qquad \psi=0,\,\,\, \textnormal{on} \,\,\Gamma;\qquad \nabla\psi\neq0 \,\,\,\,\textnormal{in}\,\,\overline{G\setminus G_1}.
		$$
  \end{lm}
		Since $\vert\nabla\psi\vert^2=\vert\nabla_\Gamma\psi\vert^2+\vert\partial_\nu\psi\vert^2$ on $\Gamma$, the function $\psi$ satisfies also that
		$$\nabla_\Gamma\psi=0\,,\qquad \vert\nabla\psi\vert=\vert\partial_\nu\psi\vert\,,\qquad\partial_\nu\psi\leq -c<0\,\,\,\textnormal{on}\,\,\Gamma, \,\,\textnormal{for some}\,\,c>0.
		$$
  
For any parameters $\lambda>1$ and $\mu>1$, we choose the following weight functions
 \begin{equation}\label{3.2}
\begin{array}{l}
\alpha(t,x) = (t(T-t))^{-1}\,(e^{\mu\psi(x)}-e^{2\mu\vert\psi\vert_\infty}),\quad \theta=e^{\lambda\alpha},\\\\
\hspace{1.6cm}\varphi(t,x) = (t(T-t))^{-1}\,e^{\mu\psi(x)},
\end{array}
\end{equation}
 for $x\in\overline{G}$ and $t\in(0,T)$. Moreover, the weights are constant on the boundary $\Gamma$ so that
\begin{equation}\label{3.3010}
\nabla_\Gamma\alpha=0\qquad\textnormal{and}\qquad\nabla_\Gamma\varphi=0\qquad\textnormal{on}\,\,\Gamma.
\end{equation}
It is easy to check that for some suitable constant $C>0$ depending only on $G$ and $G_0$, for all $(t,x)\in\overline{Q}$, we have
\begin{align}\label{3.3}
\begin{aligned}
&\,\varphi(t,x)\geq CT^{-2},\qquad\vert\varphi_t(t,x)\vert\leq CT\varphi^2(t,x),\qquad\vert\varphi_{tt}(t,x)\vert\leq CT^2\varphi^3(t,x),\\
&\vert\alpha_t(t,x)\vert\leq CTe^{2\mu\vert\psi\vert_\infty}\varphi^2(t,x),\qquad\vert\alpha_{tt}(t,x)\vert\leq CT^2e^{2\mu\vert\psi\vert_\infty}\varphi^3(t,x).
		\end{aligned}
	\end{align}
Recalling the known Carleman estimate for deterministic parabolic equations with dynamic boundary conditions (see, e.g., \cite{lipsh, khoutaibi2020null}), one can conclude the following Carleman estimate for the random parabolic equation \eqref{3.1} with $F\equiv F_1\equiv F_\Gamma\equiv F_{1,\Gamma}\equiv0$.
\begin{lm}\label{lm3.2}
There exist constants $C>0$ and $\mu_0$ depending only on $G$, $G_0$, $A$ and $A_\Gamma$ such that for all $F_0\in L^2_\mathcal{F}(0,T;L^2(G))$, $F_{0,\Gamma}\in L^2_\mathcal{F}(0,T;L^2(\Gamma))$ and $(z_0,z_{\Gamma,0})\in L^2_{\mathcal{F}_0}(\Omega;\mathbb{L}^2)$, the weak solution $(z,z_\Gamma)\in L^2_\mathcal{F}(\Omega;C([0,T];\mathbb{L}^2))\cap L^2_\mathcal{F}(0,T;\mathbb{H}^1)$ of \eqref{3.1} (with $F\equiv F_1\equiv F_\Gamma\equiv F_{1,\Gamma}\equiv0$) satisfies that  
\begin{align}
\begin{aligned}
&\,\lambda^3\mu^4\mathbb{E}\int_Q \theta^2\varphi^3z^2\,dxdt + \lambda^3\mu^3\mathbb{E}\int_\Sigma \theta^2\varphi^3z_\Gamma^2\,d\sigma dt\\
&+\lambda\mu^2\mathbb{E}\int_Q \theta^2\varphi \vert\nabla z\vert^2\,dxdt + \lambda\mu\mathbb{E}\int_\Sigma \theta^2\varphi \vert\nabla_\Gamma z_\Gamma\vert^2 \,d\sigma dt\\
&\label{3.40}\leq C \Bigg[ \lambda^3\mu^4\mathbb{E}\int_{Q_0} \theta^2\varphi^3 z^2 \,dxdt + \mathbb{E}\int_Q \theta^2 F_0^2 \,dxdt+ \mathbb{E}\int_\Sigma \theta^2 F_{0,\Gamma}^2 \,d\sigma dt\Bigg],
\end{aligned}
\end{align}
for all $\mu\geq\mu_0$ and $\lambda\geq C(T+T^2)$.
\end{lm}
Throughout this subsection, we set $\mu=\mu_0$ and $\lambda\geq C(T+T^2)$ given in Lemma \ref{lm3.2} such that the Carleman estimate \eqref{3.40} holds. Now, let us consider the following controlled backward stochastic parabolic equation
	\begin{equation}\label{3.5}
		\begin{cases}
			\begin{array}{ll}
				dr + \textnormal{div}(A\nabla r) \,dt = (\lambda^3\theta^2\varphi^3z + \mathbbm{1}_{G_0}v) \,dt + R_1 \,dW(t) &\textnormal{in}\,\,Q,\\
				dr_\Gamma+\textnormal{div}_\Gamma(A_\Gamma\nabla_\Gamma r_\Gamma) \,dt-\partial^A_\nu r \,dt = \lambda^3\theta^2\varphi^3z_\Gamma \,dt+R_2 \,dW(t) &\textnormal{on}\,\,\Sigma,\\
				r_\Gamma(t,x)=r\vert_\Gamma(t,x) &\textnormal{on}\,\,\Sigma,\\
				(r,r_\Gamma)\vert_{t=T}=(0,0) &\textnormal{in}\,\,G\times\Gamma,
			\end{array}
		\end{cases}
	\end{equation}
	where $(z,z_\Gamma)\in L^2_\mathcal{F}(\Omega;C([0,T];\mathbb{L}^2))\cap L^2_\mathcal{F} (0,T;\mathbb{H}^1)$ is the weak solution of \eqref{3.1}, $v\in L^2_\mathcal{F}(0,T;L^2(G_0))$ is the control function and $(r,r_\Gamma,R_1,R_2)$ is the state variable.

We have the following controllability result of system \eqref{3.5} which will be the key tool to establish our global Carleman estimate for equation \eqref{3.1}.
\begin{prop}\label{prop3.1}
There exists a control $\hat{v}\in L^2_\mathcal{F}(0,T;L^2(G_0))$ such that the corresponding solution $(\hat{r},\hat{r}_\Gamma,\hat{R_1},\hat{R_2})\in \big(L^2_\mathcal{F}(\Omega;C([0,T];\mathbb{L}^2))\cap L^2_\mathcal{F}(0,T;\mathbb{H}^1)\big)\times L^2_\mathcal{F}(0,T;\mathbb{L}^2)$ to \eqref{3.5} satisfies that 
$$(\hat{r}(0,\cdot),\hat{r}_\Gamma(0,\cdot))=(0,0)\;\, \textnormal{in} \;\,G\times\Gamma,\quad\mathbb{P}\textnormal{-a.s}.$$
Moreover, there exists $C=C(G,G_0,\mu_0,A,A_\Gamma)>0$ such that 
\begin{align}
    \begin{aligned}
&\,\lambda^{-3} \mathbb{E}\int_{Q} \theta^{-2}\varphi^{-3}\hat{v}^2 \,dxdt +\mathbb{E}\int_Q \theta^{-2}\hat{r}^2 \,dxdt+\mathbb{E}\int_\Sigma \theta^{-2}\hat{r}_\Gamma^2 \,d\sigma dt\\
&+\lambda^{-2}\mathbb{E}\int_Q \theta^{-2}\varphi^{-2}\vert\nabla\hat{r}\vert^2\,dxdt+\lambda^{-2}\mathbb{E}\int_\Sigma \theta^{-2}\varphi^{-2}\vert\nabla_\Gamma\hat{r}_\Gamma\vert^2\,d\sigma dt\\
  &+\lambda^{-2}\mathbb{E}\int_Q \theta^{-2}\varphi^{-2}\hat{R}_1^2\,dxdt+\lambda^{-2}\mathbb{E}\int_\Sigma \theta^{-2}\varphi^{-2}\hat{R}_2^2\,d\sigma dt \\
  &\label{3.6}
\leq C \Bigg[\lambda^3 \mathbb{E}\int_Q \theta^2\varphi^3 z^2 \,dxdt+\lambda^3\mathbb{E}\int_\Sigma \theta^2\varphi^3 z_\Gamma^2 \,d\sigma dt\Bigg],
       \end{aligned}
\end{align}
for all $\lambda\geq C(T+T^2)$.
\end{prop}
\begin{proof}
The idea of the proof using the so-called penalized Hilbert Uniqueness method, which is introduced in \cite{GloLions}. We first construct a family of optimal approximate-null controls for equation \eqref{3.5} and get a uniform estimate for approximate controls and solutions w.r.t the source terms. Hence, by using some limit arguments, we prove the controllability result with the desired estimate. This method called the duality method and it was used successfully in \cite{imanuvilov2003carleman} with Dirichlet boundary conditions. Let $\varepsilon>0$, and consider the weight function 
$$\alpha_\varepsilon\equiv\alpha_\varepsilon(t,x) = ((t+\varepsilon)(T-t+\varepsilon)) ^{-1}\,(e^{\mu_0\psi(x)}-e^{2\mu_0\vert\psi\vert_\infty}).$$ 
Put $\theta_\varepsilon=e^{\lambda\alpha_\varepsilon}$, and introduce the following minimization problem 
\begin{equation}\label{3.770}
    \inf\{J_\varepsilon(v)\,,\,\,\,v\in\mathcal{V}\},
\end{equation}
where
\begin{align*}
\begin{aligned}
J_\varepsilon(v)=&\,\frac{1}{2}\mathbb{E}\int_Q \theta_\varepsilon^{-2} r^2 \,dxdt+\frac{1}{2}\mathbb{E}\int_\Sigma \theta_\varepsilon^{-2} r_\Gamma^2 \,d\sigma dt+\frac{1}{2}\mathbb{E}\int_{Q_0} \lambda^{-3}\theta^{-2}\varphi^{-3}v^2 \,dxdt\\
&+\frac{1}{2\varepsilon}\mathbb{E}\int_G r^2(0) dx+\frac{1}{2\varepsilon}\mathbb{E}\int_\Gamma r_\Gamma^2(0) d\sigma,
\end{aligned}
\end{align*}
and 
$$\mathcal{V}=\bigg\{v\in L^2_\mathcal{F}(0,T;L^2(G_0))\,,\,\,\quad\mathbb{E}\int_{Q_0} \theta^{-2}\varphi^{-3}v^2 \,dxdt<\infty\bigg\}.$$
It is easy to see that the functional $J_\varepsilon$ is well-defined, continuous, strictly convex, and coercive. Then, the minimization problem \eqref{3.770} admits a unique optimal solution $v_\varepsilon$. By the Euler-Lagrange equation and the duality system (see, e.g., \cite{lions1972some}), the control $v_\varepsilon$ can be characterized as
		\begin{equation}\label{3.8}
			v_\varepsilon = \mathbbm{1}_{G_0}\lambda^3\theta^2\varphi^3q_\varepsilon\,,\quad \textnormal{in}\,\,\,Q\,,\,\,\,\mathbb{P}\textnormal{-a.s},\end{equation}
where $(q_\varepsilon,q_{\varepsilon,\Gamma})$ is the solution of the following random 
        parabolic equation
		\begin{equation}\label{3.98}
			\begin{cases}
				\begin{array}{ll}
					dq_\varepsilon - \textnormal{div}(A\nabla q_\varepsilon) \,dt = \theta_\varepsilon^{-2}r_\varepsilon \,dt  &\textnormal{in}\,\,Q,\\
					d q_{\varepsilon,\Gamma}-\textnormal{div}_\Gamma(A_\Gamma\nabla_\Gamma q_{\varepsilon,\Gamma}) \,dt+\partial_\nu^A q_\varepsilon \,dt = \theta_\varepsilon^{-2} r_{\varepsilon,\Gamma} \,dt &\textnormal{on}\,\,\Sigma,\\
					q_{\varepsilon,\Gamma}(t,x)=q_{\varepsilon}\vert_\Gamma(t,x) &\textnormal{on}\,\,\Sigma,\\
					(q_{\varepsilon},q_{\varepsilon,\Gamma})\vert_{t=0}=\big(\frac{1}{\varepsilon}r_\varepsilon(0),\frac{1}{\varepsilon}r_{\varepsilon,\Gamma}(0)\big) &\textnormal{in}\,\,G\times\Gamma,
				\end{array}
			\end{cases}
		\end{equation}
		where $(r_\varepsilon,r_{\varepsilon,\Gamma},R_{\varepsilon,1},R_{\varepsilon,2})$ is the solution of \eqref{3.5} with control $v_\varepsilon$.\\
By Itô's formula, computing $d\langle(r_\varepsilon,r_{\varepsilon,\Gamma}),(q_\varepsilon,q_{\varepsilon,\Gamma})\rangle_{\mathbb{L}^2}$, integrating the equality on $(0,T)$, taking the mean value on both sides and recalling \eqref{3.8}, we conclude that
\begin{align}
    \begin{aligned}
&\,\lambda^{3}\mathbb{E}\int_{Q_0}\theta^{2}\varphi^{3} q_\varepsilon^2\,dxdt+\mathbb{E}\int_Q \theta_\varepsilon^{-2} r_\varepsilon^2 \,dxdt +\mathbb{E}\int_\Sigma \theta_\varepsilon^{-2} r_{\varepsilon,\Gamma}^2\,d\sigma dt\\
&+ \frac{1}{\varepsilon}\mathbb{E}\int_G r^2_\varepsilon(0)dx + \frac{1}{\varepsilon}\mathbb{E}\int_\Gamma r^2_{\varepsilon,\Gamma}(0)d\sigma\\
&\label{3.105}
= -\lambda^3\mathbb{E}\int_Q \theta^2\varphi^3q_\varepsilon z \,dxdt-\lambda^3\mathbb{E}\int_\Sigma \theta^2\varphi^3q_{\varepsilon,\Gamma} z_\Gamma \,d\sigma dt.
    \end{aligned}
\end{align}
Applying Young's inequality, \eqref{3.105} implies that for all $\rho>0$
\begin{align}
    \begin{aligned}
&\,\lambda^{3}\mathbb{E}\int_{Q_0}\theta^{2}\varphi^{3} q_\varepsilon^2\,dxdt+\mathbb{E}\int_Q \theta_\varepsilon^{-2} r_\varepsilon^2 \,dxdt +\mathbb{E}\int_\Sigma \theta_\varepsilon^{-2} r_{\varepsilon,\Gamma}^2\,d\sigma dt\\
&+ \frac{1}{\varepsilon}\mathbb{E}\int_G r^2_\varepsilon(0)dx + \frac{1}{\varepsilon}\mathbb{E}\int_\Gamma r^2_{\varepsilon,\Gamma}(0)d\sigma\\
&\label{3.11}
\leq\rho\Bigg[\lambda^3\mathbb{E}\int_Q \theta^2\varphi^3q_\varepsilon^2 \,dxdt+\lambda^3\mathbb{E}\int_\Sigma \theta^2\varphi^3q_{\varepsilon,\Gamma}^2 \,d\sigma dt\Bigg]\\
&\quad\,+C(\rho)\Bigg[\lambda^3\mathbb{E}\int_Q \theta^2\varphi^3z^2 \,dxdt+\lambda^3\mathbb{E}\int_\Sigma \theta^2\varphi^3 z_\Gamma^2 \,d\sigma dt\Bigg].
    \end{aligned}
\end{align}
Using Carleman estimate \eqref{3.40} for solutions of \eqref{3.98} and noting that $\theta^2\theta_\varepsilon^{-2}\leq1$, then \eqref{3.11} provides that 
\begin{align*}
    \begin{aligned}
&\,\lambda^{3}\mathbb{E}\int_{Q_0}\theta^{2}\varphi^{3} q_\varepsilon^2\,dxdt+\mathbb{E}\int_Q \theta_\varepsilon^{-2} r_\varepsilon^2 \,dxdt +\mathbb{E}\int_\Sigma \theta_\varepsilon^{-2} r_{\varepsilon,\Gamma}^2\,d\sigma dt\\
&+ \frac{1}{\varepsilon}\mathbb{E}\int_G r^2_\varepsilon(0)dx + \frac{1}{\varepsilon}\mathbb{E}\int_\Gamma r^2_{\varepsilon,\Gamma}(0)d\sigma\\
&
\leq\rho\Bigg[\lambda^{3}\mathbb{E}\int_{Q_0}\theta^{2}\varphi^{3} q_\varepsilon^2\,dxdt+\mathbb{E}\int_Q \theta_\varepsilon^{-2} r_\varepsilon^2 \,dxdt +\mathbb{E}\int_\Sigma \theta_\varepsilon^{-2} r_{\varepsilon,\Gamma}^2\,d\sigma dt\Bigg]\\
&\quad\,+C(\rho)\Bigg[\lambda^3\mathbb{E}\int_Q \theta^2\varphi^3z^2 \,dxdt+\lambda^3\mathbb{E}\int_\Sigma \theta^2\varphi^3 z_\Gamma^2 \,d\sigma dt\Bigg],
\end{aligned}
\end{align*}
for any $\lambda\geq C(T+T^2)$. Hence, by choosing a small enough $\rho>0$, it follows that
\begin{align}
    \begin{aligned}
&\,\lambda^{-3}\mathbb{E}\int_{Q}\theta^{-2}\varphi^{-3} v_\varepsilon^2\,dxdt+\mathbb{E}\int_Q \theta_\varepsilon^{-2} r_\varepsilon^2 \,dxdt +\mathbb{E}\int_\Sigma \theta_\varepsilon^{-2} r_{\varepsilon,\Gamma}^2\,d\sigma dt\\
&+ \frac{1}{\varepsilon}\mathbb{E}\int_G r^2_\varepsilon(0)dx + \frac{1}{\varepsilon}\mathbb{E}\int_\Gamma r^2_{\varepsilon,\Gamma}(0)d\sigma\\
&\label{3.120}
\leq C\Bigg[\lambda^3\mathbb{E}\int_Q \theta^2\varphi^3z^2 \,dxdt+\lambda^3\mathbb{E}\int_\Sigma \theta^2\varphi^3z_{\Gamma}^2 \,d\sigma dt\Bigg].
       \end{aligned}
\end{align}
On the other hand, using again Itô's formula, we compute
$$\mathbb{E}\int_Q d(\theta_\varepsilon^{-2}\varphi^{-2}r_\varepsilon^2)dx+\mathbb{E}\int_\Sigma d(\theta_\varepsilon^{-2}\varphi^{-2}r_{\varepsilon,\Gamma}^2)d\sigma.$$
Then we deduce that
\begin{align}
    \begin{aligned}
    &\,2\mathbb{E}\int_Q \theta^{-2}_\varepsilon\varphi^{-2}A\nabla r_\varepsilon\cdot\nabla r_\varepsilon \,dxdt+\mathbb{E}\int_Q \theta^{-2}_\varepsilon\varphi^{-2}R_{\varepsilon,1}^2\,dxdt\\
    &+2\mathbb{E}\int_\Sigma \theta^{-2}_\varepsilon\varphi^{-2}A_\Gamma\nabla_\Gamma r_{\varepsilon,\Gamma}\cdot\nabla_\Gamma r_{\varepsilon,\Gamma}\,d\sigma dt+\mathbb{E}\int_\Sigma \theta^{-2}_\varepsilon\varphi^{-2} R^2_{\varepsilon,2}\,d\sigma dt\\
   &=-\mathbb{E}\int_Q (\theta^{-2}_\varepsilon\varphi^{-2})_tr_\varepsilon^2\,dxdt-\mathbb{E}\int_\Sigma (\theta^{-2}_\varepsilon\varphi^{-2})_tr^2_{\varepsilon,\Gamma}\,d\sigma dt\\
   &\quad\,-2\mathbb{E}\int_Q r_\varepsilon A\nabla r_\varepsilon\cdot\nabla(\theta^{-2}_\varepsilon\varphi^{-2})\,dxdt-2\lambda^3\mathbb{E}\int_Q \theta^{-2}_\varepsilon\theta^2\varphi r_\varepsilon z \,dxdt\\
   &\label{3.13}
\quad\,-2\lambda^3\mathbb{E}\int_\Sigma \theta^{-2}_\varepsilon\theta^2\varphi r_{\varepsilon,\Gamma}z_\Gamma \,d\sigma dt-2\mathbb{E}\int_{Q_0} \theta^{-2}_\varepsilon\varphi^{-2}r_\varepsilon v_\varepsilon \,dxdt.
    \end{aligned}
\end{align}
It is easy to check that for  $\lambda\geq CT^2$, we have for all $(t,x)\in(0,T)\times\widetilde{G}$
\begin{equation}\label{3.14}
\vert(\theta^{-2}_\varepsilon\varphi^{-2})_t\vert\leq CT\lambda \theta_\varepsilon^{-2},\qquad\quad\vert\nabla(\theta^{-2}_\varepsilon\varphi^{-2})\vert\leq C\lambda\theta^{-2}_\varepsilon\varphi^{-1},
\end{equation}
where $C$ depends also on $\mu_0$. Notice that $\theta_\varepsilon^{-1}\leq\theta^{-1}$, using \eqref{3.13}, \eqref{3.14}, assumptions on $A$ and $A_\Gamma$ and Young's inequality, we get for any  $\rho>0$ and any $\lambda\geq CT^2$
\begin{align}
    \begin{aligned}
    &\,\mathbb{E}\int_Q \theta^{-2}_\varepsilon\varphi^{-2} \vert\nabla r_\varepsilon\vert^2 \,dxdt+\mathbb{E}\int_Q \theta^{-2}_\varepsilon\varphi^{-2}R_{\varepsilon,1}^2\,dxdt\\
    &+\mathbb{E}\int_\Sigma \theta^{-2}_\varepsilon\varphi^{-2}\vert\nabla_\Gamma r_{\varepsilon,\Gamma}\vert^2 \,d\sigma dt+\mathbb{E}\int_\Sigma \theta^{-2}_\varepsilon\varphi^{-2} R^2_{\varepsilon,2}\,d\sigma dt\\
    &\leq CT\lambda \mathbb{E}\int_Q \theta_\varepsilon^{-2} r_\varepsilon^2 \,dxdt+CT\lambda \mathbb{E}\int_\Sigma \theta_\varepsilon^{-2} r_{\varepsilon,\Gamma}^2\,d\sigma dt\\
&\quad\,+\rho\mathbb{E}\int_Q \theta^{-2}_\varepsilon\varphi^{-2} \vert\nabla r_\varepsilon\vert^2 \,dxdt+C(\rho)\lambda^2\mathbb{E}\int_Q \theta_\varepsilon^{-2} r_\varepsilon^2 \,dxdt\\
&\quad\,+C\lambda^2\mathbb{E}\int_Q \theta_\varepsilon^{-2} r_\varepsilon^2 \,dxdt+CT^2\lambda^4\mathbb{E}\int_Q \theta^{2}\varphi^3 z^2 \,dxdt\\
&\label{3.1501}
\quad\,+C\lambda^2\mathbb{E}\int_\Sigma \theta_\varepsilon^{-2} r_{\varepsilon,\Gamma}^2 \,d\sigma dt+CT^2\lambda^4\mathbb{E}\int_\Sigma \theta^{2}\varphi^3 z_\Gamma^2 \,d\sigma dt\\
&\quad\,+C\lambda^2\mathbb{E}\int_Q \theta_\varepsilon^{-2} r_{\varepsilon}^2 \,dxdt+CT^2\lambda^{-2}\mathbb{E}\int_{Q}\theta^{-2}\varphi^{-3} v_\varepsilon^2 \,dxdt.
    \end{aligned}
\end{align}
By choosing a small enough $\rho>0$ \eqref{3.1501}, then multiplying the obtained inequality by $\lambda^{-2}$, we obtain for a large $\lambda\geq C(T+T^2)$
\begin{align}
    \begin{aligned}
&\,\lambda^{-2}\mathbb{E}\int_Q \theta^{-2}_\varepsilon\varphi^{-2} \vert\nabla r_\varepsilon\vert^2 \,dxdt+\lambda^{-2}\mathbb{E}\int_Q \theta^{-2}_\varepsilon\varphi^{-2}R_{\varepsilon,1}^2\,dxdt\\
&+\lambda^{-2}\mathbb{E}\int_\Sigma \theta^{-2}_\varepsilon\varphi^{-2}\vert\nabla_\Gamma r_{\varepsilon,\Gamma}\vert^2 \,d\sigma dt+\lambda^{-2}\mathbb{E}\int_\Sigma \theta^{-2}_\varepsilon\varphi^{-2} R^2_{\varepsilon,2}\,d\sigma dt\\
&\leq C\mathbb{E}\int_Q \theta_\varepsilon^{-2} r_\varepsilon^2 \,dxdt +\mathbb{E}\int_\Sigma \theta_\varepsilon^{-2} r_{\varepsilon,\Gamma}^2\,d\sigma dt+C\lambda^3\mathbb{E}\int_Q \theta^2\varphi^3z^2 \,dxdt\\
&\label{3.1502}
\quad\,+C\lambda^3\mathbb{E}\int_\Sigma \theta^2\varphi^3z_{\Gamma}^2 \,d\sigma dt+C\lambda^{-3}\mathbb{E}\int_{Q}\theta^{-2}\varphi^{-3} v_\varepsilon^2\,dxdt.
    \end{aligned}
\end{align}
Combining \eqref{3.120} and \eqref{3.1502}, we obtain
\begin{align}
    \begin{aligned}
&\,\lambda^{-3}\mathbb{E}\int_{Q}\theta^{-2}\varphi^{-3} v_\varepsilon^2\,dxdt+\mathbb{E}\int_Q \theta_\varepsilon^{-2} r_\varepsilon^2 \,dxdt +\mathbb{E}\int_\Sigma \theta_\varepsilon^{-2} r_{\varepsilon,\Gamma}^2\,d\sigma dt\\
&+ \lambda^{-2}\mathbb{E}\int_Q \theta_\varepsilon^{-2}\varphi^{-2}\vert\nabla r_\varepsilon\vert^2 \,dxdt+\lambda^{-2}\mathbb{E}\int_\Sigma \theta_\varepsilon^{-2}\varphi^{-2}\vert\nabla_\Gamma r_{\varepsilon,\Gamma}\vert^2 \,d\sigma dt\\
&+\lambda^{-2}\mathbb{E}\int_Q \theta_\varepsilon^{-2}\varphi^{-2}R_{\varepsilon,1}^2\,dxdt+\lambda^{-2}\mathbb{E}\int_\Sigma \theta_\varepsilon^{-2}\varphi^{-2}R_{\varepsilon,2}^2\,d\sigma dt\\
&+\frac{1}{\varepsilon}\mathbb{E}\int_G r^2_\varepsilon(0)dx + \frac{1}{\varepsilon}\mathbb{E}\int_\Gamma r^2_{\varepsilon,\Gamma}(0)d\sigma\\
&\label{3.15}
\leq C\Bigg[ \lambda^3\mathbb{E}\int_Q \theta^2\varphi^3z^2 \,dxdt+ \lambda^3\mathbb{E}\int_\Sigma \theta^2\varphi^3z_\Gamma^2 \,d\sigma dt\Bigg],
    \end{aligned}
\end{align}
for any $\lambda\geq C(T+T^2)$. Then, it follows that there exists  
$$(\hat{v},\hat{r},\hat{r}_\Gamma,\hat{R}_1,\hat{R}_2)\in L^2_\mathcal{F}(0,T;L^2(G_0))\times L^2_\mathcal{F}(0,T;H^1(G))\times L^2_\mathcal{F}(0,T;H^1(\Gamma))\times L^2_\mathcal{F}(0,T;\mathbb{L}^2),$$
such that as $\varepsilon\rightarrow0$,
\begin{align}\label{3.16}
\begin{aligned}
&\,v_\varepsilon \longrightarrow \hat{v},\,\,\,\textnormal{weakly}\,\,\,\textnormal{in}\,\,\,L^2((0,T)\times\Omega;L^2(G_0));\\
&r_\varepsilon \longrightarrow \hat{r},\,\,\,\textnormal{weakly}\,\,\,\textnormal{in}\,\,\,L^2((0,T)\times\Omega;H^1(G));\\
&r_{\varepsilon,\Gamma} \longrightarrow \hat{r}_\Gamma,\,\,\,\textnormal{weakly}\,\,\,\textnormal{in}\,\,\,L^2((0,T)\times\Omega;H^1(\Gamma));\\
&R_{\varepsilon,1} \longrightarrow \hat{R}_1,\,\,\,\textnormal{weakly}\,\,\,\textnormal{in}\,\,\,L^2((0,T)\times\Omega;L^2(G));\\
&R_{\varepsilon,2} \longrightarrow\hat{R}_2,\,\,\,\textnormal{weakly}\,\,\,\textnormal{in}\,\,\,L^2((0,T)\times\Omega;L^2(\Gamma)).
\end{aligned}
\end{align}
Let us now show that $(\hat{r},\hat{r}_\Gamma,\hat{R}_1,\hat{R}_2)$ is the solution to \eqref{3.5} associated to the control $\hat{v}$. To prove this fact, let us consider that $(\Tilde{r},\Tilde{r}_\Gamma,\Tilde{R}_1,\Tilde{R}_2)$ is the unique solution in $\big(L^2_\mathcal{F}(\Omega;C([0,T];\mathbb{L}^2))\cap L^2_\mathcal{F}(0,T;\mathbb{H}^1)\big)\times L^2_\mathcal{F}(0,T;\mathbb{L}^2)$ to \eqref{3.5} with the control $\hat{v}$. For any processes $f_1,f_2\in L^2_\mathcal{F}(0,T;L^2(G))$ and $g_1,g_2\in L^2_\mathcal{F}(0,T;L^2(\Gamma))$, we consider the following stochastic parabolic equation
		\begin{equation}\label{3.126}
			\begin{cases}
				\begin{array}{lcl}
					d\phi-\textnormal{div}(A\nabla\phi) \,dt = f_1 \,dt+f_2 \,dW(t) &\textnormal{in}& Q,\\
					d\phi_\Gamma-\textnormal{div}_\Gamma(A_\Gamma\nabla_\Gamma\phi_\Gamma) dt+\partial_\nu^A \phi \,dt = g_1 \,dt+g_2 \,dW(t) &\textnormal{on}& \Sigma,\\
\phi_\Gamma(t,x)=\phi\vert_\Gamma(t,x)&\textnormal{on}&\Sigma,\\
					(\phi,\phi_\Gamma)\vert_{t=0}=(0,0)&\textnormal{in}&G\times\Gamma.
				\end{array}
			\end{cases}
		\end{equation}
Using Itô's formula, we compute ``$d\langle(\phi,\phi_\Gamma),(\Tilde{r},\Tilde{r}_\Gamma)\rangle_{\mathbb{L}^2}-d\langle(\phi,\phi_\Gamma),(r_\varepsilon,r_{\varepsilon,\Gamma})\rangle_{\mathbb{L}^2}$'', and  letting $\varepsilon\rightarrow0$, we get
\begin{align*}
&\,\mathbb{E}\int_Q [(\Tilde{r}-\hat{r})f_1+(\Tilde{R}_1-\hat{R}_1)f_2] \,dxdt\\
&+\mathbb{E}\int_\Sigma [(\Tilde{r}_\Gamma-\hat{r}_\Gamma)g_1+(\Tilde{R}_2-\hat{R}_2)g_2] \,d\sigma dt = 0.
\end{align*}
Then, it follows that $\Tilde{r}=\hat{r}$ and $\Tilde{R}_1=\hat{R}_1$ in $Q$, $\mathbb{P}\textnormal{-a.s.}$, $\Tilde{r}_\Gamma=\hat{r}_\Gamma$ and $\Tilde{R}_2=\hat{R}_2$ on $\Sigma$, $\mathbb{P}\textnormal{-a.s.}$ Then we deduce that $(\hat{r},\hat{r}_\Gamma,\hat{R}_1,\hat{R}_2)$ is the solution of \eqref{3.5} with the control $\hat{v}$. Finally, combining \eqref{3.15} and \eqref{3.16}, we conclude the controllability result of \eqref{3.5} and the desired estimate \eqref{3.6}. This completes the proof of Proposition \ref{prop3.1}.
\end{proof}
Now, we are in a position to provide our main global Carleman estimate for system \eqref{3.1}.
\begin{thm}\label{thm3.1}
For $\mu=\mu_0$ given in Lemma \ref{lm3.2}, there exists a positive constant $C$ depending on $G$, $G_0$, $\mu_0$, $A$ and $A_\Gamma$ such that for all $F_0, F_1\in L^2_\mathcal{F}(0,T;L^2(G))$, $F\in L^2_\mathcal{F}(0,T;L^2(G,\mathbb{R}^N))$, $F_{0,\Gamma}, F_{1,\Gamma}\in L^2_\mathcal{F}(0,T;L^2(\Gamma))$, $F_\Gamma\in L^2_\mathcal{F}(0,T;L^2(\Gamma;\mathbb{R}^N))$ and $(z_0,z_{\Gamma,0})\in L^2_{\mathcal{F}_0}(\Omega;\mathbb{L}^2)$, the weak solution $(z,z_\Gamma)$ of equation \eqref{3.1} satisfies that 
\begin{align}\label{car3.6}
    \begin{aligned}
&\,\lambda^3\mathbb{E}\int_Q \theta^2\varphi^3z^2\,dxdt + \lambda^3\mathbb{E}\int_\Sigma \theta^2\varphi^3z_\Gamma^2\,d\sigma dt\\
&+\lambda\mathbb{E}\int_Q \theta^2\varphi \vert\nabla z\vert^2\,dxdt + \lambda\mathbb{E}\int_\Sigma \theta^2\varphi \vert\nabla_\Gamma z_\Gamma\vert^2 \,d\sigma dt\\
&\leq C \Bigg[ \lambda^3\mathbb{E}\int_{Q_0} \theta^2\varphi^3 z^2 \,dxdt + \mathbb{E}\int_Q \theta^2 F_0^2 \,dxdt \\
&\qquad\,\,+\lambda^2\mathbb{E}\int_Q \theta^2\varphi^2F_1^2 \,dxdt+\lambda^2\mathbb{E}\int_Q \theta^2\varphi^2 \vert F\vert^2 \,dxdt+\mathbb{E}\int_\Sigma \theta^2 F_{0,\Gamma}^2\,d\sigma dt\\
&\qquad\,\,+\lambda^2\mathbb{E}\int_\Sigma\theta^2\varphi^2F_{1,\Gamma}^2 \,d\sigma dt+\lambda^2\mathbb{E}\int_\Sigma\theta^2\varphi^2 \vert F_\Gamma\vert^2 \,d\sigma dt\Bigg],
\end{aligned}
\end{align}
for all $\lambda\geq C(T+T^2)$.
\end{thm}
\begin{proof}
Let $(z,z_\Gamma)$ be the solution of \eqref{3.1} and $(\hat{r},\hat{r}_\Gamma,\hat{R}_1,\hat{R}_2)$ be the solution of \eqref{3.5} with $v=\hat{v}$ given in Proposition \ref{prop3.1}. By applying Itô's formula in Lemma \eqref{lm1.1} and using also Lemma \ref{lm1.2}, we obtain that
\begin{align*}
&\,\lambda^3\mathbb{E}\int_Q \theta^2\varphi^3z^2 \,dxdt+\lambda^3\mathbb{E}\int_\Sigma \theta^2\varphi^3z_\Gamma^2\,d\sigma dt\\
&=-\mathbb{E}\int_Q [\mathbbm{1}_{G_0}z\hat{v}+F_0\hat{r}+F_1\hat{R}_1-F\cdot\nabla\hat{r}]\,dxdt\\
&\quad\,-\mathbb{E}\int_\Sigma [F_{0,\Gamma}\hat{r}_\Gamma+F_{1,\Gamma}\hat{R}_2-F_\Gamma\cdot\nabla_\Gamma\hat{r}_\Gamma]\,d\sigma dt.
\end{align*}
By Young's inequality, it follows that for all $\rho>0$, we have
\begin{align}
    \begin{aligned}
&\,\lambda^3\mathbb{E}\int_Q \theta^2\varphi^3z^2 \,dxdt+\lambda^3\mathbb{E}\int_\Sigma \theta^2\varphi^3z_\Gamma^2 \,d\sigma dt\\
&\leq\rho\Bigg[\lambda^{-3} \mathbb{E}\int_{Q} \theta^{-2}\varphi^{-3}\hat{v}^2 \,dxdt +\mathbb{E}\int_Q \theta^{-2}\hat{r}^2 \,dxdt+\mathbb{E}\int_\Sigma \theta^{-2}\hat{r}_\Gamma^2 \,d\sigma dt\\
&\quad\,+\lambda^{-2}\mathbb{E}\int_Q \theta^{-2}\varphi^{-2}\vert\nabla\hat{r}\vert^2\,dxdt+\lambda^{-2}\mathbb{E}\int_\Sigma \theta^{-2}\varphi^{-2}\vert\nabla_\Gamma\hat{r}_\Gamma\vert^2\,d\sigma dt\\
&\quad\,+\lambda^{-2}\mathbb{E}\int_Q \theta^{-2}\varphi^{-2}\hat{R}_1^2\,dxdt+\lambda^{-2}\mathbb{E}\int_\Sigma \theta^{-2}\varphi^{-2}\hat{R}_2^2\,d\sigma dt \Bigg]\\
&\quad\,+C(\rho)\Bigg[ \lambda^3\mathbb{E}\int_{Q_0} \theta^2\varphi^3 z^2 \,dxdt + \mathbb{E}\int_Q \theta^2 F_0^2 \,dxdt +\lambda^2\mathbb{E}\int_Q \theta^2\varphi^2F_1^2 \,dxdt\\
&\qquad\qquad\;\;+\lambda^2\mathbb{E}\int_Q \theta^2\varphi^2 \vert F\vert^2 \,dxdt+\mathbb{E}\int_\Sigma \theta^2 F_{0,\Gamma}^2\,d\sigma dt\\
&\label{3.240123}\qquad\qquad\;\;+\lambda^2\mathbb{E}\int_\Sigma\theta^2\varphi^2F_{1,\Gamma}^2 \,d\sigma dt+\lambda^2\mathbb{E}\int_\Sigma\theta^2\varphi^2 \vert F_\Gamma\vert^2 \,d\sigma dt\Bigg]
\end{aligned}
\end{align}
Using inequality \eqref{3.6} in \eqref{3.240123} and taking a small enough $\rho>0$, we conclude that 
\begin{align}
    \begin{aligned}
&\,\lambda^3\mathbb{E}\int_Q \theta^2\varphi^3z^2 \,dxdt+\lambda^3\mathbb{E}\int_\Sigma \theta^2\varphi^3z_\Gamma^2 \,d\sigma dt\\
&\leq C\Bigg[ \lambda^3\mathbb{E}\int_{Q_0} \theta^2\varphi^3 z^2 \,dxdt + \mathbb{E}\int_Q \theta^2 F_0^2 \,dxdt +\lambda^2\mathbb{E}\int_Q \theta^2\varphi^2F_1^2 \,dxdt\\
&\label{3.200}
\quad\quad\,+\lambda^2\mathbb{E}\int_Q \theta^2\varphi^2 \vert F\vert^2 \,dxdt+\mathbb{E}\int_\Sigma \theta^2 F_{0,\Gamma}^2\,d\sigma dt+\lambda^2\mathbb{E}\int_\Sigma\theta^2\varphi^2F_{1,\Gamma}^2 \,d\sigma dt\\
&\quad\quad\,+\lambda^2\mathbb{E}\int_\Sigma\theta^2\varphi^2 \vert F_\Gamma\vert^2 \,d\sigma dt\Bigg],
\end{aligned}
\end{align}
for all $\lambda\geq C(T+T^2)$.

On the other hand, differentiating $d(\theta^2\varphi z,z)_{L^2(G)}$ by using Itô's formula, we get
\begin{align}\label{3.20}
    \begin{aligned}
0=&\,\mathbb{E}\int_Q (\theta^2\varphi)_tz^2\,dxdt-2\mathbb{E}\int_Q A\nabla z\cdot\nabla(\theta^2\varphi z)\,dxdt\\
&+2\mathbb{E}\int_0^T\langle\partial_\nu^A z,\theta^2\varphi z_\Gamma\rangle_{H^{-1/2}(\Gamma),H^{1/2}(\Gamma)} dt+2\mathbb{E}\int_Q \theta^2\varphi zF_0\,dxdt\\
&-2\mathbb{E}\int_Q F\cdot\nabla(\theta^2\varphi z)\,dxdt+2\mathbb{E}\int_0^T \langle F\cdot\nu,\theta^2\varphi z_\Gamma\rangle_{H^{-1/2}(\Gamma),H^{1/2}(\Gamma)} dt\\
&+\mathbb{E}\int_Q \theta^2\varphi F_1^2\,dxdt,
\end{aligned}
\end{align}
and by the same way, we compute also $d(\theta^2\varphi z_\Gamma,z_\Gamma)_{L^2(\Gamma)}$, we obtain
\begin{align}\label{3.21}
    \begin{aligned}
0=&\,\mathbb{E}\int_\Sigma (\theta^2\varphi)_tz_\Gamma^2\,d\sigma dt-2\mathbb{E}\int_\Sigma A_\Gamma\nabla_\Gamma z_\Gamma\cdot\nabla_\Gamma(\theta^2\varphi z_\Gamma)\,d\sigma dt\\
&-2\mathbb{E}\int_0^T \langle \partial_\nu^A z,\theta^2\varphi z_\Gamma\rangle_{H^{-1/2}(\Gamma),H^{1/2}(\Gamma)} dt+2\mathbb{E}\int_\Sigma\theta^2\varphi z_\Gamma F_{0,\Gamma} \,d\sigma dt\\
&-2\mathbb{E}\int_0^T \langle F\cdot\nu,\theta^2\varphi z_\Gamma\rangle_{H^{-1/2}(\Gamma),H^{1/2}(\Gamma)} dt-2\mathbb{E}\int_\Sigma F_\Gamma\cdot\nabla_\Gamma(\theta^2\varphi z_\Gamma)\,d\sigma dt\\
&+\mathbb{E}\int_\Sigma \theta^2\varphi F_{1,\Gamma}^2\,d\sigma dt.
\end{aligned}
\end{align}
Adding \eqref{3.20} and \eqref{3.21}, it is easy to see that
\begin{align}
    \begin{aligned}
&\,2\mathbb{E}\int_Q \theta^2\varphi A\nabla z\cdot\nabla z \,dxdt+2\mathbb{E}\int_\Sigma \theta^2\varphi A_\Gamma\nabla_\Gamma z_\Gamma\cdot\nabla_\Gamma z_\Gamma \,d\sigma dt\\
&=\mathbb{E}\int_Q (\theta^2\varphi)_tz^2\,dxdt-2
\mathbb{E}\int_Q zA\nabla z\cdot\nabla(\theta^2\varphi) \,dxdt
\\
&\quad\,-2\mathbb{E}\int_Q \theta^2\varphi F\cdot\nabla z \,dxdt-2\mathbb{E}\int_Q z\nabla(\theta^2\varphi)\cdot F \,dxdt
+2\mathbb{E}\int_Q \theta^2\varphi zF_0\,dxdt\\
&\quad\,+\mathbb{E}\int_Q \theta^2\varphi F_1^2\,dxdt+\mathbb{E}\int_\Sigma (\theta^2\varphi)_tz_\Gamma^2\,d\sigma dt
+2\mathbb{E}\int_\Sigma\theta^2\varphi z_\Gamma F_{0,\Gamma} \,d\sigma dt\\
&\label{3.22}
\quad\,-2\mathbb{E}\int_\Sigma \theta^2\varphi F_\Gamma\cdot\nabla_\Gamma z_\Gamma \,d\sigma dt+\mathbb{E}\int_\Sigma \theta^2\varphi F_{1,\Gamma}^2\,d\sigma dt.
\end{aligned}
\end{align}
Now, it is easy to check that for $\lambda\geq CT^2$,
\begin{equation}\label{3.23}
    \vert (\theta^2\varphi)_t\vert\leq CT\lambda \theta^2\varphi^3,\,\,\,\textnormal{in}\,\,\overline{Q};\qquad \vert\nabla(\theta^2\varphi)\vert\leq C\lambda\theta^2\varphi^2,\,\,\,\textnormal{in}\,\,Q,
\end{equation}
where $C$ depends on $\mu_0$ a well. Combining \eqref{3.22}, \eqref{3.23} and using assumptions on $A$ and $A_\Gamma$, we end up with
\begin{align}
    \begin{aligned}
&\,\mathbb{E}\int_Q \theta^2\varphi\vert\nabla z\vert^2\,dxdt+\mathbb{E}\int_\Sigma \theta^2\varphi\vert\nabla_\Gamma z_\Gamma\vert^2\,d\sigma dt\\
&\leq CT\lambda \mathbb{E}\int_Q \theta^2\varphi^3z^2\,dxdt+C\lambda\mathbb{E}\int_Q \theta^2\varphi^2\vert z\vert\vert\nabla z \vert \,dxdt\\
&\quad\,+C\mathbb{E}\int_Q  \theta^2\varphi\vert F\vert\vert\nabla z\vert \,dxdt+C\lambda\mathbb{E}\int_Q \theta^2\varphi^2\vert F\vert\vert z\vert \,dxdt+C\mathbb{E}\int_Q \theta^2\varphi\vert F_0\vert\vert z\vert \,dxdt\\
&\quad\,+C\mathbb{E}\int_Q \theta^2\varphi F_1^2\,dxdt+CT\lambda \mathbb{E}\int_\Sigma \theta^2\varphi^3 z_\Gamma^2\,d\sigma dt+C\mathbb{E}\int_\Sigma \theta^2\varphi\vert F_{0,\Gamma}\vert\vert z_\Gamma\vert \,d\sigma dt\\
&\label{3.24}
\quad\,+\mathbb{E}\int_\Sigma \theta^2\varphi\vert F_\Gamma\vert\vert\nabla_\Gamma z_\Gamma\vert \,d\sigma dt+C\mathbb{E}\int_\Sigma\theta^2\varphi F_{1,\Gamma}^2\,d\sigma dt.
\end{aligned}
\end{align}
Applying Young's inequality, \eqref{3.24} provides that for all $\rho>0$, one has
\begin{align}
    \begin{aligned}
&\,\mathbb{E}\int_Q \theta^2\varphi\vert\nabla z\vert^2\,dxdt+\mathbb{E}\int_\Sigma \theta^2\varphi\vert\nabla_\Gamma z_\Gamma\vert^2\,d\sigma dt\\
&\leq\rho\Bigg[\mathbb{E}\int_Q \theta^2\varphi\vert\nabla z\vert^2\,dxdt+\mathbb{E}\int_\Sigma \theta^2\varphi\vert\nabla_\Gamma z_\Gamma\vert^2\,d\sigma dt\Bigg]\\
&\quad\,+C(\rho)\Bigg[T\lambda \mathbb{E}\int_Q\theta^2\varphi^3z^2\,dxdt+\lambda^2\mathbb{E}\int_Q\theta^2\varphi^3z^2\,dxdt\\
&\qquad\qquad\,\;+\mathbb{E}\int_Q\theta^2\varphi\vert F\vert^2\,dxdt+\lambda\mathbb{E}\int_Q\theta^2\varphi^2z^2\,dxdt+\lambda^{-1}\mathbb{E}\int_Q\theta^2F_0^2\,dxdt\\
&\label{3.25}
\qquad\qquad\,\;+\mathbb{E}\int_Q\theta^2\varphi F_1^2\,dxdt+T\lambda \mathbb{E}\int_\Sigma\theta^2\varphi^3z_\Gamma^2\,d\sigma dt+\lambda\mathbb{E}\int_\Sigma\theta^2\varphi^2z_\Gamma^2\,d\sigma dt\\
&\qquad\qquad\,\;+\lambda^{-1}\mathbb{E}\int_\Sigma\theta^2F_{0,\Gamma}^2\,d\sigma dt+\mathbb{E}\int_\Sigma\theta^2\varphi\vert F_\Gamma\vert^2 \,d\sigma dt+\mathbb{E}\int_\Sigma\theta^2\varphi F_{1,\Gamma}^2\,d\sigma dt\Bigg].
\end{aligned}
\end{align}
Choosing a small enough $\rho$ in \eqref{3.25}, then multiplying the obtained inequality by $\lambda$ and recalling \eqref{3.3}, we deduce that
\begin{align}
    \begin{aligned}
&\,\lambda\mathbb{E}\int_Q \theta^2\varphi\vert\nabla z\vert^2\,dxdt+\lambda\mathbb{E}\int_\Sigma \theta^2\varphi\vert\nabla_\Gamma z_\Gamma\vert^2\,d\sigma dt\\
&\leq C\Bigg[T\lambda^2 \mathbb{E}\int_Q\theta^2\varphi^3z^2\,dxdt+\lambda^3\mathbb{E}\int_Q\theta^2\varphi^3z^2\,dxdt\\
&\quad\quad\,+T^2\lambda\mathbb{E}\int_Q\theta^2\varphi^2\vert F\vert^2\,dxdt+T^2\lambda^2\mathbb{E}\int_Q\theta^2\varphi^3z^2\,dxdt+\mathbb{E}\int_Q\theta^2F_0^2\,dxdt\\
&\quad\quad\,+T^2\lambda\mathbb{E}\int_Q\theta^2\varphi^2 F_1^2\,dxdt+T\lambda^2\mathbb{E}\int_\Sigma\theta^2\varphi^3z_\Gamma^2\,d\sigma dt+T^2\lambda^2\mathbb{E}\int_\Sigma\theta^2\varphi^3z_\Gamma^2\,d\sigma dt\\
&\label{3.2501}
\quad\quad\,+\mathbb{E}\int_\Sigma\theta^2F_{0,\Gamma}^2\,d\sigma dt+T^2\lambda\mathbb{E}\int_\Sigma\theta^2\varphi^2\vert F_\Gamma\vert^2 \,d\sigma dt+T^2\lambda\mathbb{E}\int_\Sigma\theta^2\varphi^2 F_{1,\Gamma}^2\,d\sigma dt\Bigg].
\end{aligned}
\end{align}
Finally, combining \eqref{3.200} and \eqref{3.2501}, and taking a large $\lambda\geq C(T+T^2)$, we get the desired Carleman estimate \eqref{car3.6}. This concludes the proof of Theorem \ref{thm3.1}.
\end{proof}
\subsection{Proof of Theorem \ref{thm1.2}}
Firstly, we show the following global Carleman estimate for solutions of the adjoint equation \eqref{1.012} which is an easy consequence of Carleman estimate \eqref{car3.6}.
\begin{cor}\label{corcarl}
There exists $C=C(G,G_0,\mu_0,A,A_\Gamma)>0$ such that the weak solution $(z,z_\Gamma)$ of system \eqref{1.012} satisfies that
\begin{align}
\begin{aligned}
&\,\lambda^3\mathbb{E}\int_Q \theta^2 \varphi^3z^2 \,dxdt +\lambda^3\mathbb{E}\int_\Sigma \theta^2\varphi^3z^2_\Gamma \,d\sigma dt\\
&+\lambda\mathbb{E}\int_Q \theta^2\varphi\vert\nabla z\vert^2 \,dxdt+\lambda\mathbb{E}\int_\Sigma \theta^2\varphi\vert\nabla_\Gamma z_\Gamma\vert^2\,d\sigma dt\\
&\label{1.4}
\leq C\lambda^3\,\mathbb{E}\int_{Q_0} \theta^2\varphi^3 z^2 \,dxdt,
\end{aligned}
\end{align}
for all $\lambda\geq\lambda_1= C\big[T+T^2(1+\vert a_1\vert_\infty^{2/3}+\vert a_2\vert_\infty^{2}+\vert B\vert_\infty^2+\vert b_1\vert_\infty^{2/3}+\vert b_2\vert_\infty^2+\vert B_\Gamma\vert_\infty^2)\big].$  
\end{cor}
\begin{proof} By choosing coefficients in the Carleman estimate \eqref{car3.6} as follows:
\begin{align*}
&\,F_0=-a_1z,\qquad F_1=-a_2z,\qquad F=zB,\\ &F_{0,\Gamma}=-b_1z_\Gamma,\qquad F_{1,\Gamma}=-b_2z_\Gamma,\qquad F_\Gamma=z_\Gamma B_\Gamma,
\end{align*}
we obtain that
\begin{align}
    \begin{aligned}
&\,\lambda^3\mathbb{E}\int_Q \theta^2\varphi^3z^2\,dxdt + \lambda^3\mathbb{E}\int_\Sigma \theta^2\varphi^3z_\Gamma^2\,d\sigma dt\\
&+\lambda\mathbb{E}\int_Q \theta^2\varphi \vert\nabla z\vert^2\,dxdt+ \lambda\mathbb{E}\int_\Sigma \theta^2\varphi \vert\nabla_\Gamma z_\Gamma\vert^2 \,d\sigma dt\\
&\leq C\lambda^3\mathbb{E}\int_{Q_0} \theta^2\varphi^3 z^2 \,dxdt + C\mathbb{E}\int_Q \theta^2 \vert a_1z\vert^2 \,dxdt\\
& \quad\,+C\lambda^2\mathbb{E}\int_Q \theta^2\varphi^2\vert a_2z\vert^2 \,dxdt+C\lambda^2\mathbb{E}\int_Q \theta^2\varphi^2 \vert zB\vert^2 \,dxdt\\
  &\label{3.3001}
\quad\,+C\mathbb{E}\int_\Sigma \theta^2 \vert b_1z_\Gamma\vert^2\,d\sigma dt+C\lambda^2\mathbb{E}\int_\Sigma\theta^2\varphi^2\vert b_2z_\Gamma\vert^2 \,d\sigma dt\\
&\quad\,+C\lambda^2\mathbb{E}\int_\Sigma\theta^2\varphi^2 \vert z_\Gamma B_\Gamma \vert^2 \,d\sigma dt,
  \end{aligned}
\end{align}
for any $\lambda\geq C(T+T^2)$. Now, note that it is sufficient to choose $$\lambda\geq CT^2 \,(\vert a_1\vert_\infty^{2/3}+\vert a_2\vert_\infty^{2}+\vert B\vert_\infty^2+\vert b_1\vert_\infty^{2/3}+\vert b_2\vert_\infty^2+\vert B_\Gamma\vert_\infty^2),$$		
to get that
\begin{align}
    \begin{aligned}
&\,\,C\mathbb{E}\int_Q \theta^2 \vert a_1z\vert^2 \,dxdt +C\lambda^2\mathbb{E}\int_Q \theta^2\varphi^2\vert a_2z\vert^2 \,dxdt\\
 &+C\lambda^2\mathbb{E}\int_Q \theta^2\varphi^2 \vert Bz\vert^2 \,dxdt+C\mathbb{E}\int_\Sigma \theta^2 \vert b_1z_\Gamma\vert^2\,d\sigma dt\\
&+C\lambda^2\mathbb{E}\int_\Sigma\theta^2\varphi^2\vert b_2z_\Gamma\vert^2 \,d\sigma dt+C\lambda^2\mathbb{E}\int_\Sigma\theta^2\varphi^2 \vert B_\Gamma z_\Gamma\vert^2 d\sigma\\
&\label{3.3101}
\leq \frac{1}{2}\lambda^3\mathbb{E}\int_Q \theta^2\varphi^3z^2 \,dxdt+\frac{1}{2}\lambda^3\mathbb{E}\int_\Sigma \theta^2\varphi^3z_\Gamma^2 \,d\sigma dt.
\end{aligned}
\end{align}
Finally, combining \eqref{3.3001} and \eqref{3.3101}, we deduce the desired Carleman estimate \eqref{1.4} for any $\lambda\geq\lambda_1$ which is defined by
$$\lambda_1=C\big[T+T^2+T^2(\vert a_1\vert_\infty^{2/3}+\vert a_2\vert_\infty^{2}+\vert B\vert_\infty^2+\vert b_1\vert_\infty^{2/3}+\vert b_2\vert_\infty^2+\vert B_\Gamma\vert_\infty^2)\big].$$
\end{proof}

Now, we are ready to prove the observability inequality of \eqref{1.012}.
\begin{proof}[Proof of Theorem \ref{thm1.2}]
From Carleman estimate \eqref{1.4}, we deduce
\begin{align}\label{5.2}
\begin{aligned}		    
&\,\mathbb{E}\int_{T/4}^{3T/4}\int_G \theta^2\varphi^3 z^2(t,x) \,dxdt + \mathbb{E}\int_{T/4}^{3T/4}\int_\Gamma \theta^2\varphi^3 z_\Gamma^2(t,x) \,d\sigma dt\\
&\leq C \, \mathbb{E}\int_{Q_0} \theta^2\varphi^3 z^2(t,x) \,dxdt,
\end{aligned}
\end{align}
 for any $\lambda\geq\lambda_1$. It is also easy to see that for $(t,x)\in (T/4,3T/4)\times\widetilde{G}$
\begin{equation}\label{4.301}
    e^{2\lambda_1\alpha}\varphi^3 \geq CT^{-6}e^{-2C\big(1+\frac{1}{T}+\vert a_1\vert_\infty^{2/3}+\vert a_2\vert_\infty^{2}+\vert B\vert_\infty^2+\vert b_1\vert_\infty^{2/3}+\vert b_2\vert_\infty^2+\vert B_\Gamma\vert_\infty^2\big)},
\end{equation}
and for $(t,x)\in(0,T)\times G$
\begin{equation}\label{4.401}
    e^{2\lambda_1\alpha}\varphi^3 \leq CT^{-6}e^{-C\big(1+\frac{1}{T}+\vert a_1\vert_\infty^{2/3}+\vert a_2\vert_\infty^{2}+\vert B\vert_\infty^2+\vert b_1\vert_\infty^{2/3}+\vert b_2\vert_\infty^2+\vert B_\Gamma\vert_\infty^2\big)}.
\end{equation}
Fix $\lambda=\lambda_1$ and combining \eqref{5.2}, \eqref{4.301} and \eqref{4.401}, we conclude that
 \begin{align}\label{4.5012}
 \begin{aligned}
     &\,\mathbb{E}\int_{T/4}^{3T/4}\int_G z^2(t,x) \,dxdt + \mathbb{E}\int_{T/4}^{3T/4}\int_\Gamma z_\Gamma^2(t,x) \,d\sigma dt\\
     &\leq e^{C\big(1+\frac{1}{T}+r_1\big)} \, \mathbb{E}\int_{Q_0} z^2(t,x) \,dxdt,
      \end{aligned}
 \end{align}
 where $r_1=\vert a_1\vert_\infty^{2/3}+\vert a_2\vert_\infty^{2}+\vert B\vert_\infty^2+\vert b_1\vert_\infty^{2/3}+\vert b_2\vert_\infty^2+\vert B_\Gamma\vert_\infty^2$.
 
On the other hand, let $t\in[0,T]$, by Itô's formula we compute $d_s\langle(z,z_\Gamma),(z,z_\Gamma)\rangle_{\mathbb{L}^2}$, integrating the equality on $[t,T]$ and taking the expectation, we obtain that
\begin{align}
    \begin{aligned}
&\,\mathbb{E}\int_G  z^2(T,x) dx + \mathbb{E}\int_\Gamma  z_\Gamma^2(T,x) d\sigma -\left(\mathbb{E}\int_G  z^2(t,x) dx + \mathbb{E}\int_\Gamma  z_\Gamma^2(t,x) d\sigma\right) \\
&\leq -2\beta\mathbb{E}\int_t^T\int_G \vert\nabla z\vert^2dxds+C(\vert a_1\vert_\infty+\vert a_2\vert_\infty^2)\mathbb{E}\int_t^T\int_G z^2dxds\\
&\label{4.610}
\quad\,+C\vert B\vert_\infty\mathbb{E}\int_t^T\int_G \vert z\vert\vert\nabla z\vert dxds-2\beta\mathbb{E}\int_t^T\int_\Gamma \vert\nabla_\Gamma z_\Gamma\vert^2d\sigma ds\\
&\quad\,+C(\vert b_1\vert_\infty+\vert b_2\vert_\infty^2)\mathbb{E}\int_t^T\int_\Gamma z_\Gamma^2d\sigma ds+C\vert B_\Gamma\vert_\infty\mathbb{E}\int_t^T\int_\Gamma \vert z_\Gamma\vert\vert\nabla_\Gamma z_\Gamma\vert d\sigma ds.
    \end{aligned}
\end{align}
Applying Young's inequality on the right-hand side of \eqref{4.610}, one can absorb gradient terms in the third and sixth terms using respectively the first and fourth terms, then we deduce that
\begin{align*}
    \begin{aligned}
&\,\mathbb{E}\int_G  z^2(T,x) dx + \mathbb{E}\int_\Gamma  z_\Gamma^2(T,x) d\sigma\\
&\leq \mathbb{E}\int_G  z^2(t,x) dx + \mathbb{E}\int_\Gamma  z_\Gamma^2(t,x) d\sigma+Cr_2\,\Bigg[\mathbb{E}\int_t^T\int_G  z^2 dx ds+ \mathbb{E}\int_t^T\int_\Gamma  z_\Gamma^2 d\sigma ds\Bigg].
    \end{aligned}
\end{align*}
where $r_2=\vert a_1\vert_\infty+\vert a_2\vert_\infty^2+\vert B\vert_\infty^2+\vert b_1\vert_\infty+\vert b_2\vert_\infty^2+\vert B_\Gamma\vert_\infty^2$. Hence, by Gronwall inequality, it follows that for all $t\in[0,T]$
 \begin{equation}\label{4.810}
\mathbb{E}\int_G  z^2(T,x) dx + \mathbb{E}\int_\Gamma  z_\Gamma^2(T,x) d\sigma
\leq e^{CTr_2}\,\Bigg[\mathbb{E}\int_G  z^2(t,x) dx + \mathbb{E}\int_\Gamma  z_\Gamma^2(t,x) d\sigma\Bigg].
 \end{equation}
Finally, integrating \eqref{4.810} on $(T/4,3T/4)$ and combining the obtained inequality with \eqref{4.5012}, it is straightforward to obtain the observability estimate \eqref{1.3}. This concludes the proof of Theorem \ref{thm1.2}.
\end{proof}
\section{Proof of Theorem \ref{thm1.1}}\label{sec4}
This section is addressed to give the proof of our main result about null controllability of \eqref{1.1} i.e., Theorem \ref{thm1.1}.
\begin{proof}[Proof of Theorem \ref{thm1.1}]
Let us fix $\varepsilon>0$ and $(y_T,y_{\Gamma,T})\in L^2_{\mathcal{F}_T}(\Omega;\mathbb{L}^2)$ and consider the following optimal problem
\begin{equation}\label{4.90123}
    \inf\{J_\varepsilon(u)\,,\,\,\,u\in L^2_\mathcal{F}(0,T;L^2(G))\},
\end{equation}
where
\begin{align*}
\begin{aligned}
J_\varepsilon(u)=\frac{1}{2}\mathbb{E}\int_{Q_0} u^2 \,dxdt+\frac{1}{2\varepsilon}\mathbb{E}\int_G y^2(0) dx+\frac{1}{2\varepsilon}\mathbb{E}\int_\Gamma y_\Gamma^2(0) d\sigma,
\end{aligned}
\end{align*}
where $(y,y_\Gamma,Y,\widetilde{Y})$ is the solution of system \eqref{1.1} with the control $u$ and the terminal state $(y_T,y_{\Gamma,T})$. Firstly, it is easy to see that the functional $J_\varepsilon$ is well-defined, continuous, strictly convex and coercive, then the problem \eqref{4.90123} has a unique optimal solution $u_\varepsilon$. Combining the Euler equation (i.e., Fréchet derivative $J'_\varepsilon(u_\varepsilon)=0$) and the optimality system (see, e.g., \cite{lions1972some,liu2014global}), it is easy to see that $u_\varepsilon$ can be characterized as follows
\begin{equation}\label{4.2452}
u_\varepsilon=-\mathbbm{1}_{G_0}z_\varepsilon
\end{equation}
where $(z_\varepsilon,z_{\varepsilon,\Gamma})$ is the solution of the following forward stochastic parabolic equation
\begin{equation}\label{1.2}
{\small\begin{cases}
			\begin{array}{ll}
				dz_\varepsilon - \textnormal{div}(A\nabla z_\varepsilon) \,dt = (-a_1 z_\varepsilon+\textnormal{div}(z_\varepsilon B)) \,dt - a_2z_\varepsilon \,dW(t) &\textnormal{in}\,\,Q,\\
				dz_{\varepsilon,\Gamma}-\textnormal{div}_\Gamma(A_\Gamma \nabla_\Gamma z_{\varepsilon,\Gamma}) \,dt+\partial^A_\nu z_\varepsilon \,dt = (-b_1z_{\varepsilon,\Gamma}-z_{\varepsilon,\Gamma} B\cdot \nu+\textnormal{div}_\Gamma(z_{\varepsilon,\Gamma} B_\Gamma))\,dt\\
    \hspace{5.3cm}\;\;\;\, -b_2z_{\varepsilon,\Gamma} \,dW(t) &\textnormal{on}\,\,\Sigma,\\
				z_{\varepsilon,\Gamma}(t,x)=z_\varepsilon\vert_\Gamma(t,x) &\textnormal{on}\,\,\Sigma,\\
				(z_\varepsilon,z_{\varepsilon,\Gamma})\vert_{t=0}=(-\frac{1}{\varepsilon}y_{\varepsilon}(0),-\frac{1}{\varepsilon}y_{\varepsilon,\Gamma}(0)) &\textnormal{in}\,\,G\times\Gamma,
			\end{array}
		\end{cases}}
\end{equation}
where $(y_\varepsilon,y_{\varepsilon,\Gamma},Y_\varepsilon,\widetilde{Y}_\varepsilon)$ is the solution of \eqref{1.1} with the control $u_\varepsilon$ and the terminal state $(y_T,y_{\Gamma,T})$. Now, by using Itô's formula in Lemma \eqref{lm1.1}, we compute $d\langle(y_\varepsilon,y_{\varepsilon,\Gamma}),(z_\varepsilon,z_{\varepsilon,\Gamma})\rangle_{\mathbb{L}^2}$, integrating the result on $(0,T)$ and taking the expectation on both sides, we end up with
\begin{align*}
    \begin{aligned}
        &\,-\mathbb{E}\int_Q \mathbbm{1}_{G_0}u_\varepsilon z_\varepsilon \,dxdt+\frac{1}{\varepsilon}\mathbb{E}\int_G y^2_\varepsilon(0,\cdot) \,dx+\frac{1}{\varepsilon}\mathbb{E}\int_\Gamma y^2_{\varepsilon,\Gamma}(0,\cdot) \,d\sigma\\
        &=-\mathbb{E}\int_G y_T z_\varepsilon(T,\cdot)\,dx-\mathbb{E}\int_\Gamma y_{\Gamma,T}z_{\varepsilon,\Gamma}(T,\cdot)\,d\sigma.
    \end{aligned}
\end{align*}
By Cauchy-Schwarz and Young's inequalities, it follows that
\begin{align}\label{4.343}
    \begin{aligned}
        &\,-\mathbb{E}\int_Q \mathbbm{1}_{G_0}u_\varepsilon z_\varepsilon \,dxdt+\frac{1}{\varepsilon}\mathbb{E}\int_G y^2_\varepsilon(0,\cdot) \,dx+\frac{1}{\varepsilon}\mathbb{E}\int_\Gamma y^2_{\varepsilon,\Gamma}(0,\cdot) \,d\sigma\\
        &\leq \frac{e^{CK}}{2}\vert(y_T,y_{\Gamma,T})\vert_{L^2_{\mathcal{F}_T}(\Omega;\mathbb{L}^2)}^2+\frac{1}{2e^{CK}}\big(\vert z_\varepsilon(T,\cdot)\vert^2_{L^2_{\mathcal{F}_T}(\Omega;L^2(G))}+\vert z_{\varepsilon,\Gamma}(T,\cdot)\vert^2_{L^2_{\mathcal{F}_T}(\Omega;L^2(\Gamma))}\big),
    \end{aligned}
\end{align}
where $e^{CK}$ is the same constant appeared in \eqref{1.3}. Therefore, recalling \eqref{4.2452} and combining \eqref{4.343} and \eqref{1.3}, we conclude that
\begin{align}\label{4.120123}
    \begin{aligned}
        \vert u_\varepsilon\vert^2_{L^2(0,T;L^2(G))}+\frac{2}{\varepsilon}\vert(y_\varepsilon(0,\cdot),y_{\varepsilon,\Gamma}(0,\cdot))\vert^2_{L^2_{\mathcal{F}_0}(\Omega;\mathbb{L}^2)}
        \leq e^{CK}\,\vert(y_T,y_{\Gamma,T})\vert_{L^2_{\mathcal{F}_T}(\Omega;\mathbb{L}^2)}^2.
    \end{aligned}
\end{align}
From \eqref{4.120123}, one can extract a subsequence (denoted also by $u_\varepsilon$) of $u_\varepsilon$ such that 
$$u_\varepsilon \longrightarrow \hat{u},\,\,\,\textnormal{weakly in}\,\,\,\,L^2_\mathcal{F}(0,T;L^2(G)), \,\,\textnormal{as}\,\,\varepsilon\rightarrow0.$$
Let $(\hat{y},\hat{y}_\Gamma,\hat{Y},\hat{\widetilde{Y}})$ be the solution of \eqref{1.1} associated to the terminal state $(y_T,y_{\Gamma,T})$ and the control $\hat{u}$. By using again \eqref{4.120123}, we deduce $(\hat{y}(0,\cdot),\hat{y}_{\Gamma}(0,\cdot))=(0,0)$ in $G\times\Gamma$, $\mathbb{P}\textnormal{-a.s.}$, and the control $\hat{u}$ satisfies the desired estimate \eqref{1.21203}. This concludes the proof of Theorem \ref{thm1.1}.
\end{proof}

\end{document}